\documentclass[12pt]{amsart}
\usepackage{amssymb, verbatim}
\usepackage{amsthm}
\usepackage{epsfig,color}
\usepackage{enumerate}
\usepackage{color,url}
\usepackage{graphicx}
\usepackage{subcaption}
\usepackage{float}
\usepackage{placeins}
\usepackage{tikz}
\usepackage{xcolor}
\usepackage{cite}
\usetikzlibrary{arrows.meta}

\date{\today}

\newtheorem{theorem}{Theorem}[section]

\newtheorem{proposition}[theorem]{Proposition}
\theoremstyle{definition}

\theoremstyle{remark}
\newtheorem{remark}{Remark}

\makeatletter
\@namedef{subjclassname@2020}{%
 \textup{2020} Mathematics Subject Classification}
\makeatother

\numberwithin{equation}{section}

\usepackage[T1]{fontenc}

\newcommand{\hyper}[5]{\,{}_{#1}F_{#2}\left(\!\!%
\begin{array}{cc}{\displaystyle{#3}}\\[-0.1ex]
{\displaystyle{#4}} \end{array}\Big| \,{\displaystyle{#5}}
\right)}

\begin{document}
\title{Finite orthogonal polynomials on a cone}

\author[\"{O}.F. Et]{\"{O}mer Faruk Et}
\address[\"{O}.F. Et]{Ankara University, Graduate School of Natural and Applied Sciences, Department of Mathematics, 06110, D\i\c{s}kap\i, Ankara, T\"{u}rkiye; Karamanoglu Mehmetbey University, Faculty of Kamil \"{O}zdag Science, Department of Mathematics, 70100, Karaman,T\"{u}rkiye}
\email[\"{O}.F. Et]{omerfaruket68@gmail.com, omerfaruket@kmu.edu.tr}

\author[E. \c{C}ekirdek]{Esra \c{C}ek\.{I}rdek}
\address[E. \c{C}ekirdek]{Ankara University, Graduate School of Natural and Applied Sciences, Department of Mathematics, 06110, D\i\c{s}kap\i, Ankara, T\"{u}rkiye}
\email[E. \c{C}ekirdek]{ecekirdek@ankara.edu.tr}

\author[R. Akta\c{s} Karaman]{Rab\.{I}a  Akta\c{s} Karaman}
\address[R. Akta\c{s} Karaman]{Ankara University, Faculty of Science, Department of Mathematics, 06100, Tando\u{g}an, Ankara, T\"{u}rkiye}
\email[R. Akta\c{s} Karaman (corresponding author)]{raktas@science.ankara.edu.tr}

\thanks{}
\date{\today}

\begin{abstract}
The aim of this paper is to study finite orthogonal polynomials on a cone of
revolution and its surface. We define two classes of finite orthogonal
polynomials on the solid cone and derive their corresponding differential
equations and recurrence relations. Furthermore, we demonstrate that, in the
limit case, one of these classes reduces to Laguerre polynomials on the cone.
Similarly, we establish two families of finite orthogonal polynomials on the
surface of the cone and analyze their respective properties.

\end{abstract}
\maketitle

\section{Introduction}

It is well known that the classical families of orthogonal polynomials,
namely the Jacobi, Laguerre, Hermite polynomials, are infinitely orthogonal
and satisfy a second order differential equation of the form \cite{Bochner}%
\begin{equation*}
\left( ax^{2}+bx+c\right) y_{n}^{\prime \prime }+\left( dx+e\right)
y_{n}^{\prime }-n\left( \left( n-1\right) a+d\right) y_{n}=0
\end{equation*}%
where the parameters $a,b,c,d,$ and $e$ are independent of $n$. Beyond the
classical infinitely orthogonal families, there also exist three families of
finitely orthogonal which arise as special solutions of the same
second order differential equation. Such families of hypergeometric
polynomials denoted by $M_{n}^{(p,q)},$ $N_{n}^{(p)}$ and $I_{n}^{(p)}$, are
finitely orthogonal with respect to the F sampling, inverse Gamma, and T
sampling distributions. Their application in functions approximation and
numerical integration have been studied in detail in \cite{M1}. Each family
maintains a close relationship with the Jacobi and Laguerre polynomials.
Notably, the first family is related to Jacobi polynomials, while the second family is directly connected to the generalized Bessel
polynomials and, consequently, to the Laguerre polynomials \cite{M1}. The
finite orthogonal polynomials are less known in the existing literature
compared to classical infinite sequences. But, over the past few years,
these polynomials have received significant attention in the literature. For
literature regarding finite orthogonal polynomials, refer to \cite{E1, EA,
GA, GAM1, GAM2, GCO, M2, MSA, SMA}.

Orthogonal polynomials in several variables play a fundamental role in
approximation theory, harmonic analysis, and spectral methods for partial
differential equations. While the classical orthogonal polynomials in one
variable such as Jacobi, Laguerre, and Hermite polynomials are well
understood, extending orthogonality to multivariate and non-product domains
remains a challenging and active area of research. In classical settings
such as the unit sphere, the unit ball, and the simplex, the existence of
explicit orthogonal bases, the Fourier series of orthogonal polynomials,
differential equations with polynomial coefficients, cubature formulae,
closed form reproducing kernels and approximation properties have led to a
rich and well developed theory (see \cite{DaiX, DX, PX, PX2, X1, X2}).

In recent years, orthogonal structures on and inside quadratic surfaces of
revolution have been studied, motivated by the theory of spherical harmonics
on the unit sphere and classical orthogonal polynomials on the unit ball
(see \cite{X3, X5, X6, X7, X8, X9}). In \cite{X3}, orthogonal structures and the
Fourier orthogonal series were studied on the solid cone
\begin{equation*}
{\mathbb{V}}^{d+1}=\left \{ (x,t):\Vert x\Vert \leq t,\, \,x\in \mathbb{R}%
^{d},\, \,0\leq t\leq b\right \}
\end{equation*}%
and on the surface of cone of revolution
\begin{equation*}
{\mathbb{V}}_{0}^{d+1}=\left \{ (x,t):\Vert x\Vert =t,\, \,x\in \mathbb{R}%
^{d},\, \,0\leq t\leq b\right \}
\end{equation*}%
where $b$ is a nonnegative real number. On the solid cone ${\mathbb{V}}%
^{d+1} $, two families of orthogonal polynomials, namely Jacobi and Laguerre
polynomials on the cone, are defined with respect to the weight function
\begin{equation*}
W_{\mu }(x,t)=w(t)(t^{2}-\Vert x\Vert ^{2})^{\mu -\frac{1}{2}},\quad \mu >-%
\tfrac{1}{2}
\end{equation*}%
where $w$ is the Jacobi weight or the Laguerre weight. It is shown that each
of these two families are eigenfunctions of a second order differential
operator with eigenvalues which depend on only total degree of the
polynomials. Analogous to the results for the solid cone, on the surface of
cone of revolution ${\mathbb{V}}_{0}^{d+1},$ two orthogonal families with
respect to $\rho \left( t\right) \mathrm{d}\sigma \left( x,t\right) $ were
defined where $\mathrm{d}\sigma $ denotes the surface measure of ${\mathbb{V}%
}_{0}^{d+1}$ and some properties were studied \cite{X3}. Subsequently, \cite%
{X5} explored Fourier orthogonal expansions using Laguerre-type weight
functions on the conic surface of revolution and the domain bounded by such
a surface. Further developments in \cite{X6} and \cite{X7} established
orthogonal polynomials and Fourier series on double cone, hyperboloid, and
paraboloid. More recently, \cite{X8} examined orthogonal polynomials for
weight functions over domains of revolution, deriving several new families.
A recent paper \cite{X9} analyzed orthogonal polynomials on a fully
symmetric planar domain that is generated by a certain triangle and their
approximation properties.

Motivated by these studies, the purpose of this paper is to study
multivariate finite orthogonal polynomials on a cone of revolution%
\begin{equation*}
{\mathbb{V}}^{d+1}=\left \{ (x,t):\Vert x\Vert \leq t,\, \,x\in \mathbb{R}%
^{d},\, \,0\leq t<\infty \right \}
\end{equation*}%
and also on the surface of the cone%
\begin{equation*}
{\mathbb{V}}_{0}^{d+1}=\left \{ (x,t):\Vert x\Vert =t,\, \,x\in \mathbb{R}%
^{d},\, \,0\leq t<\infty \right \}
\end{equation*}%
for $d\geq 2.$ There are a few papers bivariate finite orthogonal
polynomials in the literature. But for $d>2$, finite orthogonal structures
have, to the best of our knowledge, not yet been studied in the existing
literature. For the solid cone, the finite orthogonality is defined with
respect to the weight function $W_{\mu }(x,t)=w(t)(t^{2}-\Vert x\Vert
^{2})^{\mu -\frac{1}{2}},\quad \mu >-\tfrac{1}{2}$ where $w$ is a weight
function on ${\mathbb{R}}^{+}$ for the associated finite orthogonal
polynomial. On the surface on the cone, we introduce an orthogonality
relation defined with respect to the measure $\rho \left( t\right) \mathrm{d}%
\sigma \left( x,t\right) $ where $d\sigma $ is the surface measure of ${%
\mathbb{V}}_{0}^{d+1}.$ We will extend the existing framework by
establishing two classes of finite orthogonality structures on the solid
cone and on the surface of the cone drawing on univariate finite orthogonal
polynomials, spherical harmonics on the unit sphere and classical orthogonal
polynomials on the unit ball.

The paper is organized as follows. In Section 2, we recall necessary
preliminaries on classical finite orthogonal polynomials and multivariate
orthogonal polynomials on the unit ball and on the cone. Section 3 presents
two classes of finite orthogonal polynomials on the solid cone taking a cue
the first class $M_{n}^{(p,q)}\left( x\right) $ and second class $%
N_{n}^{(p)} $ of univariate finite orthogonal polynomials. It is shown that
the first finite class on the cone is a solution of second order
differential equation while the second finite class is a solution of second
order difference-differential equation. Also, recurrence relations for these
families are derived and it is shown that in the limit case, the first
finite family on the cone reduces to Laguerre polynomials on the cone. In
the last section, we derive two families of finite orthogonal polynomials on the
surface of the cone and we provide their some properties such as orthogonality property, differential
equation, and limit relation, showing that the first finite family on the surface of the cone reduces to the Laguerre
polynomials on the surface of the cone in the limit case.

\section{Preliminary}

\setcounter{equation}{0}

For our study of finite orthogonal polynomials in and on the cone, we shall
need various properties of spherical harmonics, orthogonal polynomials on
the unit ball and on the cone as well as the first and second class of
finite orthogonal polynomials on the interval $\left[ 0,\infty \right) .$ We shall
collect what will be needed in subsequent subsection.

\subsection{Finite orthogonal polynomials of one variable}

Let $w\left( x\right) $ be a nonnegative and an integrable function on the
interval $\left[ a,b\right] .$ Let $\varphi _{n}$ denote the orthogonal
polynomials of degree $n$ with respect to the weight function $w$. Then,%
\begin{equation*}
\left \langle \varphi _{n},\varphi _{m}\right \rangle _{w}=c_{w}\int
\limits_{a}^{b}\varphi _{n}\left( x\right) \varphi _{m}\left( x\right)
w\left( x\right) \mathrm{d}x=h_{n}\delta _{m,n},~m,n=0,1,\ldots
\end{equation*}%
where $\delta _{m,n}$ denotes Kronecker delta and $h_{n}$ is the norm square
of $\varphi _{n}$ and, $c_{w}$ is the normalization constant so that $\left
\langle 1,1\right \rangle _{w}=1.$ While the classical orthogonal polynomials
like Hermite, Laguerre, and Jacobi polynomials form infinite sequences,
finite sequences arise when the orthogonality relation only holds for a
finite number of terms, $n=0,1,...,N.$ The finite orthogonality is defined
under specific conditions imposed on the parameters of the weight function $%
w\left( x\right) .$\textbf{\ }The finite orthogonal polynomials $%
M_{n}^{\left( p,q\right) }\left( x\right) ,~N_{n}^{\left( p\right) }\left(
x\right) $ and $I_{n}^{\left( p\right) }\left( x\right) $ were studied by
Masjed-Jamei \cite{M1} in detail.

Here we recall two classes of finite orthogonal polynomials we will use
throughout the paper. The first class is the finite $M_{n}^{\left(
p,q\right) }\left( x\right) $ polynomials, and second class is the finite $%
N_{n}^{\left( p\right) }\left( x\right) $ polynomials.

\subsubsection{First Class of Finite OPs:}

For the first class $M_{n}^{\left( p,q\right) }\left( x\right) $ of finite
OPs$,$ the weight function is denoted by
\begin{equation*}
w_{p,q}(x)=\frac{x^{q}}{\left( 1+x\right) ^{p+q}},~0\leq x<\infty .
\end{equation*}%
Its normalization constant $c_{p,q},$ defined by $c_{p,q}\int \limits_{0}^{%
\infty }w_{p,q}(x)\mathrm{d}x=1,$ is given by
\begin{equation}
c_{p,q}=\frac{\Gamma \left( p+q\right) }{\Gamma \left( p-1\right) \Gamma
\left( q+1\right) };~~p>1,q>-1  \label{Def:normcons M}
\end{equation}
where $\Gamma \left( p\right) $ denotes the Gamma function defined by%
\begin{equation*}
\Gamma \left( p\right) =\int \limits_{0}^{\infty }x^{p-1}e^{-x}\mathrm{d}x,~%
 p >0.
\end{equation*}
The first class of the finite orthogonal polynomials $M_{n}^{\left(
p,q\right) }\left( x\right) $ is defined by the Rodrigues formula
\begin{equation}
M_{n}^{\left( p,q\right) }\left( x\right) =\left( -1\right) ^{n}\frac{\left(
1+x\right) ^{p+q}}{x^{q}}\frac{d^{n}\left( x^{n+q}\left( 1+x\right)
^{n-p-q}\right) }{dx^{n}},~n=0,1,2,...  \label{M-Rod}
\end{equation}%
and they are orthogonal with respect to the weight function $%
w_{p,q}$ satisfying%
\begin{equation*}
c_{p,q}\int \limits_{0}^{\infty }\frac{x^{q}}{\left( 1+x\right) ^{p+q}}%
M_{n}^{\left( p,q\right) }\left( x\right) M_{m}^{\left( p,q\right) }\left(
x\right) \mathrm{d}x=h_{n}^{\left( p,q\right) }\delta _{m,n},
\end{equation*}%
if and only if $p>2N+1,~q>-1,~N=\max \{m,n\}$ where%
\begin{equation}
h_{n}^{\left( p,q\right) }=\frac{n!\left(  p-n-1\right)  !\left(  p+q-1\right)  !\left(  q+1\right)
_{n}}{\left(  p-2\right)  !\left(  p-2n-1\right)  \left(  p+q-n-1\right)  !}.
\label{Def:norm M}
\end{equation}%
Some samples of these polynomials are as follows
\begin{equation*}
\left \{
\begin{array}{l}
M_{0}^{\left( p,q\right) }\left( x\right) =1 \\
M_{1}^{\left( p,q\right) }\left( x\right) =\left( p-2\right) x-\left(
q+1\right)  \\
M_{2}^{\left( p,q\right) }\left( x\right) =\left( p-4\right) \left(
p-3\right) x^{2}-2\left( p-3\right) \left( q+2\right) x+\left( q+2\right)
\left( q+1\right)  \\
M_{3}^{\left( p,q\right) }\left( x\right) =\left( p-6\right) \left(
p-5\right) \left( p-4\right) x^{3}-3\left( p-5\right) \left( p-4\right)
\left( q+3\right) x^{2} \\
\text{ \  \  \  \  \  \  \  \  \  \  \  \  \  \  \  \  \  \ }+3\left( p-4\right) \left(
q+3\right) \left( q+2\right) x-\left( q+3\right) \left( q+2\right) \left(
q+1\right)  \\
\vdots
\end{array}%
\right. .
\end{equation*}%
The polynomials $M_{n}^{\left( p,q\right) }\left( x\right) $ are solutions
of the differential equation
\begin{equation}
x\left( 1+x\right) y_{n}^{\prime \prime }\left( x\right) +\left( \left(
2-p\right) x+\left( 1+q\right) \right) y_{n}^{\prime }\left( x\right)
-n\left( n+1-p\right) y_{n}\left( x\right) =0.  \label{equ-M}
\end{equation}%
and the following recurrence relations are satisfied for $M_{n}^{\left(
p,q\right) }\left( x\right) $
\begin{eqnarray}
M_{n+1}^{\left( p,q\right) }\left( x\right)  &=&\left( \frac{\left( p-\left(
2n+1\right) \right) \left( p-\left( 2n+2\right) \right) }{\left( p-\left(
n+1\right) \right) }x\right.   \label{recurrenceM} \\
&&+\left. \frac{\left( p-\left( 2n+1\right) \right) \left( 2n\left(
n+1\right) -p\left( q+2n+1\right) \right) }{\left( p-\left( n+1\right)
\right) \left( p-2n\right) }\right) M_{n}^{\left( p,q\right) }\left(
x\right)   \notag \\
&&-\left( \frac{n\left( p-\left( 2n+2\right) \right) \left( p+q-n\right)
\left( q+n\right) }{\left( p-\left( n+1\right) \right) \left( p-2n\right) }%
\right) M_{n-1}^{\left( p,q\right) }\left( x\right)   \notag
\end{eqnarray}%
and
\begin{equation}
\frac{d}{dx}M_{n}^{\left( p,q\right) }\left( x\right) =n\left( p-\left(
n+1\right) \right) M_{n-1}^{\left( p-2,q+1\right) }\left( x\right) .
\label{Mderrel}
\end{equation}%
There is also a limit relation between $M_{n}^{\left( p,q\right) }\left(
x\right) $ and Laguerre polynomials. In this sense we can write%
\begin{equation*}
M_{n}^{\left( p,q\right) }\left( \frac{x}{p}\right) =\left( -1\right) ^{n}%
\frac{\left( 1+x/p\right) ^{p+q}}{x^{q}}\frac{d^{n}\left( x^{n+q}\left(
1+x/p\right) ^{n-p-q}\right) }{dx^{n}},
\end{equation*}%
from which, taking limit as $p\rightarrow \infty $ gives \cite{Tez}
\begin{equation}
\underset{p\rightarrow \infty }{\lim }M_{n}^{\left( p,q\right) }\left( \frac{%
x}{p}\right) =\left( -1\right) ^{n}n!L_{n}^{\left( q\right) }\left( x\right)  \label{limit}
.
\end{equation}

\subsubsection{Second Class of Finite OPs:}

For the second class $N_{n}^{(p)}(x)$ of finite OPs$,$ the weight function
is denoted by
\begin{equation*}
w_{p}(x)=x^{-p}e^{-\frac{1}{x}},~0\leq x<\infty .
\end{equation*}%
Its normalization constant $c_{p},$ defined by $c_{p}\int
\limits_{0}^{\infty }w_{p}(x)\mathrm{d}x=1,$ is given by
\begin{equation}
c_{p}=\frac{1}{\Gamma \left( p-1\right) },~p>1  \label{Def:normcons N}.
\end{equation}%
The second class of finite orthogonal polynomials $N_{n}^{(p)}(x)$ is
defined by the Rodrigues formula%
\begin{equation*}
N_{n}^{(p)}(x)=(-1)^{n}x^{p}e^{\frac{1}{x}}\frac{d^{n}\left( x^{-p+2n}e^{-%
\frac{1}{x}}\right) }{dx^{n}},~n=0,1,2,...
\end{equation*}%
and they are finitely orthogonal with respect to $w_{p}$ on the
interval $\left[ 0,\infty \right) $ if and only if $p>2N+1,$ $N=\max
\{m,n\}. $ Indeed,
\begin{equation*}
c_{p}\int \limits_{0}^{\infty }x^{-p}e^{-\frac{1}{x}%
}N_{n}^{(p)}(x)N_{m}^{(p)}(x)\mathrm{d}x=h_{n}^{\left( p\right) }\delta
_{m,n},
\end{equation*}%
is satisfied for $p>2N+1,$ $N=\max \{m,n\}$ where
\begin{equation}
h_{n}^{\left( p\right) }=\frac{n!\left( p-n-1\right) !}{\left( p-2n-1\right)
\Gamma \left( p-1\right) }.  \label{normN}
\end{equation}%
Some samples of these polynomials are given as follows
\begin{equation*}
\left \{
\begin{array}{l}
N_{0}^{\left( p\right) }\left( x\right) =1 \\
N_{1}^{\left( p\right) }\left( x\right) =\left( p-2\right) x-1 \\
N_{2}^{\left( p\right) }\left( x\right) =\left( p-4\right) \left( p-3\right)
x^{2}-2\left( p-3\right) x+1 \\
N_{3}^{\left( p\right) }\left( x\right) =\left( p-6\right) \left( p-5\right)
\left( p-4\right) x^{3}-3\left( p-5\right) \left( p-4\right) x^{2}+3\left(
p-4\right) x-1 \\
\vdots%
\end{array}%
\right. .
\end{equation*}%
The polynomials $N_{n}^{\left( p\right) }\left( x\right) $ are solutions of
the differential equation%
\begin{equation}
x^{2}y_{n}^{\prime \prime }\left( x\right) +\left( \left( 2-p\right)
x+1\right) y_{n}^{\prime }\left( x\right) -n\left( n+1-p\right) y_{n}\left(
x\right) =0  \label{Dif-N}
\end{equation}%
and they also satisfy the following recurrence relations \cite{M1}
\begin{eqnarray}
N_{n+1}^{\left( p\right) }\left( x\right) &=&\left( \frac{\left( p-\left(
2n+2\right) \right) \left( p-\left( 2n+1\right) \right) }{p-\left(
n+1\right) }x-\frac{p\left( p-\left( 2n+1\right) \right) }{\left( p-\left(
n+1\right) \right) \left( p-2n\right) }\right) N_{n}^{\left( p\right)
}\left( x\right)  \label{Nrecrel} \\
&&-\frac{n\left( p-\left( 2n+2\right) \right) }{\left( p-\left( n+1\right)
\right) \left( p-2n\right) }N_{n-1}^{\left( p\right) }\left( x\right)  \notag
\end{eqnarray}%
and
\begin{equation}
\frac{d}{dx}N_{n}^{\left( p\right) }\left( x\right) =n\left( p-\left(
n+1\right) \right) N_{n-1}^{\left( p-2\right) }\left( x\right) .
\label{Nderrel}
\end{equation}

\subsection{Spherical Harmonics}

Let $\Delta $ be the Laplace operator with $\Delta =\left( \partial
/\partial x_{1}\right) ^{2}+\ldots +\left( \partial /\partial x_{d}\right)
^{2}$ of $\mathbb{R}^{d}$ and let $\mathcal{P}_{n}^{d}$ denote the space
homogenous polynomials of $n$ in $d$ variables. A spherical harmonic $P$ of
degree $n$ is a element of $\mathcal{P}_{n}^{d}$ which satisfies $\Delta P=0$
and for $n=0,1,\ldots ,$ and $\mathcal{H}_{n}^{d}$ denotes the space of
spherical harmonics
\begin{equation*}
\mathcal{H}_{n}^{d}=\left \{ P\in \mathcal{P}_{n}^{d}:\Delta P=0\right \} .
\end{equation*}%
Then $\dim \mathcal{P}_{n}^{d}=\binom{n+d-1}{n}$ and
\begin{equation*}
\dim \mathcal{H}_{n}^{d}=\dim \mathcal{P}_{n}^{d}-\dim \mathcal{P}_{n-2}^{d}=%
\binom{n+d-1}{n}-\binom{n+d-3}{n-2}.
\end{equation*}%
If $Y\in \mathcal{P}_{n}^{d},$ then it can be written as
\begin{equation*}
Y\left( x\right) =\left \Vert x\right \Vert ^{n}Y\left( x^{\prime }\right)
,~x^{\prime }=x/\left \Vert x\right \Vert \in \mathbb{S}^{d-1},
\end{equation*}%
which shows that $Y$ is determined by its restriction on $\mathbb{S}^{d-1}.$
Spherical harmonics of different degree form an orthogonal family on the
unit sphere $\mathbb{S}^{d-1}:=\left \{ x:\text{ }\left \Vert x\right \Vert
=1\right \} $. For $n\in \mathbb{N}_{0}$, assume that the set $\left \{
Y_{l}^{n}:1\leq l\leq \mathcal{H}_{n}^{d}\right \} $ is an orthonormal basis
of $\mathcal{H}_{n}^{d}.$ Then it follows
\begin{equation*}
\frac{1}{\varpi _{d}}\int_{\mathbb{S}^{d-1}}Y_{l}^{n}\left( x\right)
Y_{l^{\prime }}^{m}\left( x\right) \mathrm{d}\sigma \left( x\right) =\delta
_{n,m}\delta _{l,l^{\prime }}
\end{equation*}%
where $\varpi _{d}$ denotes the surface area of $\mathbb{S}^{d-1}$ and is
given by%
\begin{equation*}
\varpi _{d}:=\int_{\mathbb{S}^{d-1}}\mathrm{d}\sigma =\frac{2\pi ^{d/2}}{%
\Gamma \left( d/2\right) }.
\end{equation*}%
Spherical harmonics are eigenfunctions of the differential equation \cite[%
1.4.9]{DaiX}%
\begin{equation*}
\Delta _{0}Y=-n\left( n+d-2\right) Y,~~Y\in \mathcal{H}_{n}^{d},
\end{equation*}%
where $\Delta _{0}$ denotes the Laplace-Beltrami operator on the sphere,
that is, the restriction of the Laplacian $\Delta $ to the unit sphere.

\subsection{OPs on the unit ball}

Let ${\mathsf{w}}_{\mu }$ denote the weight function
\begin{equation*}
{\mathsf{w}}_{\mu }(x)=(1-\Vert x\Vert ^{2})^{\mu -\frac{1}{2}},\quad \mu >-%
\tfrac{1}{2},
\end{equation*}%
on the unit ball ${\mathbb{B}}^{d}=\{x\in {\mathbb{R}}^{d}:\Vert x\Vert \leq
1\}.$ Classical orthogonal polynomials on the unit ball ${\mathbb{B}}^{d}$
are orthogonal with respect to the inner product%
\begin{equation*}
\left \langle f,g\right \rangle _{\mu }=b_{\mu }^{\mathbb{B}}\int_{\mathbb{B}%
^{d}}f\left( x\right) g\left( x\right) {\mathsf{w}}_{\mu }(x)\mathrm{d}x,
\end{equation*}%
where $b_{\mu }^{\mathbb{B}}$ is the normalization constant so that $\left
\langle 1,1\right \rangle _{\mu }=1.$ The normalization constant $b_{\mu }^{%
\mathbb{B}}$ is given by
\begin{equation}
b_{\mu }^{\mathbb{B}}=\frac{1}{\int_{{\mathbb{B}}^{d}}{\mathsf{w}}_{\mu }(x)%
\mathrm{d}x}=\frac{\Gamma (\mu +\frac{d+1}{2})}{\pi ^{\frac{d}{2}}\Gamma
(\mu +\frac{1}{2})}  \label{norm-ball}.
\end{equation}

For the specific case $d=1$, the associated orthogonal polynomials reduce to the classical Gegenbauer polynomials, $C_{m}^{\left( \mu\right) }$. Following \cite[p. 277, Eq. (4)]{25}, these are defined by the Gauss hypergeometric function as
\begin{equation}\label{hyper}
C_{m}^{\left(  \mu \right)  }\left(  x\right)  =\frac{\left(
2\mu\right)  _{m}}{m!} \,  \hyper{2}{1}{-m,m+2\mu}{\mu+\frac{1}{2}}{\frac{1-x}{2}},
\end{equation}
here, $_{2}F_{1}$ is the special case ($p=2,q=1$) of the generalized hypergeometric function $_{p}F_{q}$, which is given by
\begin{equation}
\hyper{p}{q}{\alpha_{1},\alpha_{2},\dots,\alpha_{p}}{\beta_{1},\ \beta_{2},\dots,\beta_{q}}{x}=
{\displaystyle \sum \limits_{m=0}^{\infty}}
\frac{\left(  \alpha_{1}\right)  _{m}\left(  \alpha_{2}\right)  _{m} \dots \left(
\alpha_{p}\right)  _{m}}{\left(  \beta_{1}\right)  _{m}\left(  \beta_{2}\right)
_{m}\dots \left(  \beta_{q}\right)  _{m}}\frac{x^{m}}{m!}\label{genhyper}.
\end{equation}
 The coefficients are expressed using the Pochhammer symbol, $\left(\alpha \right)_{m}={\displaystyle \prod \limits_{j=0}^{m-1}}\left(  \alpha+j\right) $ for $m\geq1$, with $\left(\alpha \right)_{0}=1$, where $\alpha$ is a real or complex number. The Gegenbauer polynomials are orthogonal with respect to the weight function $w(x)=\left(  1-x^{2}\right)^{\mu-\frac{1}{2}}$ over the interval $[-1,1]$. Specifically, the orthogonality relation is given by \cite[p.281, Eq. (28)]{25}
\begin{equation}
 \int \limits_{-1}^{1}
C_{m}^{\left(  \mu \right)  }\left(  x\right)  C_{n}^{\left(
\mu \right)  }\left(  x\right)\left(  1-x^{2}\right)  ^{\mu-\frac{1}{2}}dx=\frac{\left(  2\mu \right)  _{m}\Gamma \left(  \mu
+\frac{1}{2}\right)  \Gamma \left(  \frac{1}{2}\right)  }{m!\left(
m+\mu \right)  \Gamma \left( \mu \right)  }~\delta_{m,n}%
,\label{ort} \quad m,n=0,1,....
\end{equation}

Let ${\mathcal{V}}_{n}^{d}({\mathsf{w}}_{\mu })$ be the space of orthogonal
polynomials of degree $n$ for the weight function ${\mathsf{w}}_{\mu }$. Then we have
\begin{equation*}
\dim {\mathcal{V}}_{n}^{d}=\binom{n+d-1}{n}.
\end{equation*}%
Several explicit orthogonal bases of ${\mathcal{V}}_{n}^{d}({\mathsf{w}}%
_{\mu })$ can be explicitly given in terms of classical orthogonal
polynomials of one variable; we refer the reader to \cite[Chapter 5]{DX} for further details. We can index a basis
of ${\mathcal{V}}_{n}^{d}({\mathsf{w}}_{\mu })$ by $\left \{ P_{\mathbf{k}%
}^{n}:\left \vert \mathbf{k}\right \vert =n,~\mathbf{k\in }\mathbb{N}%
_{0}^{d}\right \} $. These polynomials are eigenfunctions of second order
differential equation for $u\in {\mathcal{V}}_{n}^{d}({\mathsf{w}}_{\mu })$
\begin{equation}
\left( \Delta -\left \langle x,\bigtriangledown \right \rangle ^{2}-\left(
2\mu +d-1\right) \left \langle x,\bigtriangledown \right \rangle \right)
u=-n\left( n+2\mu +d-1\right) u,  \label{diffop}
\end{equation}%
where $\bigtriangledown $ denotes the gradient operator.

An orthogonal basis for this space can be given in terms of Jacobi
polynomials and spherical harmonics. For $0\leq m\leq n/2,$ let $\left \{
Y_{l}^{n-2m}:1\leq l\leq \dim \mathcal{H}_{n-2m}^{d}\right \} $ be an
orthonormal basis of $H_{n-2m}^{d}$ and define\textbf{\ }%
\begin{equation*}
P_{l,m}^{n}\left( x\right) =P_{m}^{\left( \mu -1/2,n-2m+\frac{d-2}{2}\right)
}\left( 2\left \Vert x\right \Vert ^{2}-1\right) Y_{l,n-2m}\left( x\right) ,
\end{equation*}%
where the Jacobi polynomial $P_{n}^{({\alpha },{\beta })}$ is given by, for $%
{\alpha },{\beta }>-1$,
\begin{equation*}
P_{n}^{{\alpha },{\beta }}(t)=\frac{({\alpha }+1)_{n}}{n!}{}_{2}F_{1}\left(
\begin{matrix}
-n,n+{\alpha }+{\beta }+1 \\
{\alpha }+1%
\end{matrix}%
;\frac{1-t}{2}\right)
\end{equation*}%
in terms of the hypergeometric function and it satisfies the orthogonal
relation \cite{Sz}
\begin{equation*}
\frac{c_{{\alpha },{\beta }}}{2^{{\alpha }+{\beta }+1}}\int_{-1}^{1}P_{n}^{({%
\alpha },{\beta })}(t)P_{m}^{({\alpha },{\beta })}(t)(1-t)^{{\alpha }}(1+t)^{%
{\beta }}\mathrm{d}t=h_{n}^{({\alpha },{\beta })}\delta _{n,m},
\end{equation*}%
where $c_{{\alpha },{\beta }}=\frac{\Gamma ({\alpha }+{\beta }+2)}{\Gamma ({%
\alpha }+1)\Gamma ({\beta }+1)}$ and the norm square is given by
\begin{equation*}
h_{n}^{({\alpha },{\beta })}=\frac{({\alpha }+1)_{n}({\beta }+1)_{n}({\alpha
}+{\beta }+n+1)}{n!({\alpha }+{\beta }+2)_{n}({\alpha }+{\beta }+2n+1)}.
\end{equation*}%
Then, the set of the polynomials $\left \{ P_{l,m}^{n}:0\leq m\leq
n/2,~1\leq l\leq \dim \mathcal{H}_{n-2m}^{d}\right \} $ is an orthogonal
basis of ${\mathcal{V}}_{n}^{d}({\mathsf{w}}_{\mu }).$ For other explicit
orthogonal bases of ${\mathcal{V}}_{n}^{d}({\mathsf{w}}_{\mu }),$ refer to
\cite{DX}.

\subsection{OPs in and on a cone}

In this subsection, we recall the infinitely orthogonal polynomials,
especially Laguerre polynomials, on the solid cone and on the surface of the
cone studied in \cite{X3}.

\subsubsection{OPs on the solid cone}

Let $W_{\mu }(x,t)$ denote the weight function%
\begin{equation*}
W_{\mu }(x,t)=w(t)(t^{2}-\Vert x\Vert ^{2})^{\mu -\frac{1}{2}},\quad \mu >-%
\tfrac{1}{2}
\end{equation*}%
on the cone ${\mathbb{V}}^{d+1}=\left \{ (x,t):\Vert x\Vert \leq t,\, \,x\in
\mathbb{R}^{d},\, \,0\leq t\leq b\right \} $ where $b$ is a nonnegative real
number and $w$ is a weight function on ${\mathbb{R}^{+}}$.
Orthogonal polynomials on the solid cone ${\mathbb{V}}^{d+1}$ are orthogonal
with respect to the inner product
\begin{equation*}
{\langle }f,g{\rangle }_{\mu }:=b_{\mu }\int_{{\mathbb{V}}%
^{d+1}}f(x,t)g(x,t)W_{\mu }(x,t)\mathrm{d}x\mathrm{d}t,
\end{equation*}%
where $b_{\mu }$ is a normalization constant so that ${\langle }1,1{\rangle }%
_{\mu }=1$ and
\begin{equation}
b_{\mu }=b_{\mu }^{{\mathbb{B}}}\times b_{\mu }^{w}\quad \hbox{with}\quad
b_{\mu }^{w}=\frac{1}{\int_{0}^{\infty }t^{d+2\mu -1}w(t)\mathrm{d}t}.
\label{nor-cone}
\end{equation}%
This can be easily verified by using the separation of variables
\begin{equation*}
\int_{{\mathbb{V}}^{d+1}}f(x,t)W_{\mu }(x,t)\mathrm{d}x\mathrm{d}%
t=\int_{0}^{\infty }\int_{{\mathbb{B}}^{d}}f(ty,t)(1-\Vert y\Vert ^{2})^{\mu
-\frac{1}{2}}\mathrm{d}y\,t^{d+2\mu -1}w(t)\mathrm{d}t.
\end{equation*}%
Let ${\mathcal{V}}_{n}({\mathbb{V}}^{d+1},W_{\mu })$ be the space of
orthogonal polynomials of degree $n$ in $d+1$ variables with respect to the
inner product ${\langle }f,g{\rangle }_{\mu }$. A basis for this space can
be constructed using orthogonal polynomials on the unit ball together with a
family of univariate orthogonal polynomials.

Let ${\mathbb{P}}_{m}=\{P_{{\mathbf{k}}}:|{\mathbf{k}}|=m\}$ denote a basis
of ${\mathcal{V}}_{m}({\mathbb{B}}^{d},{\mathsf{w}}_{\mu })$ for $m\leq n$
and let $q_{n-m}^{{\alpha }}$ be an orthogonal polynomial in one variable
with respect to the weight function $t^{{\alpha }}w(t)$ on ${\mathbb{R}^{+}}$. Define
\begin{equation}
{\mathsf{Q}}_{{\mathbf{k}},n}(x,t)=q_{n-m}^{{2m+2\alpha }}(t)t^{m}P_{{%
\mathbf{k}}}\! \left( \frac{x}{t}\right) ,\qquad |{\mathbf{k}}|=m,\quad
0\leq m\leq n,  \label{eq:sQcone}
\end{equation}%
where ${\alpha }=\mu +\frac{d-1}{2}$. Then ${\mathbb{Q}}_{n}=\{{\mathsf{Q}}_{%
{\mathbf{k}},n}:|{\mathbf{k}}|=m,\,0\leq m\leq n\}$ forms a basis of ${%
\mathcal{V}}_{n}({\mathbb{V}}^{d+1},W_{\mu })$. One example of such
polynomials is Laguerre polynomials on the cone defined by (see \cite{X3})
\begin{equation*}
{\mathsf{L}}_{{\mathbf{k}},n}^{\beta ,\mu }(x,t)=L_{n-m}^{2m+2\mu +{\beta }%
+d-1}(t)t^{m}P_{{\mathbf{k}}}^{m}\left( \frac{x}{t}\right) ,\quad |{\mathbf{k%
}}|=m,\, \,0\leq m\leq n
\end{equation*}%
where the Laguerre polynomial $L_{n}^{\alpha }$ is defined by the Rodrigues
formula, for ${\alpha }>-1$,
\begin{equation*}
L_{n}^{\alpha }(t)=\frac{x^{-\alpha }e^{x}}{n!}\frac{d^{n}}{dx^{n}}%
\left( x^{n+\alpha }e^{-x}\right)
\end{equation*}%
and it satisfies the orthogonal relation
\begin{equation*}
\frac{1}{\Gamma ({\alpha }+1)}\int_{0}^{\infty }L_{n}^{\alpha
}(t)L_{m}^{\alpha }(t)t^{\alpha }e^{-t}\mathrm{d}t=\frac{({\alpha }+1)_{n}}{%
n!}\delta _{m,n}.
\end{equation*}%
The polynomials ${\mathsf{L}}_{{\mathbf{k}},n}^{\beta ,\mu }(x,t)$ are
orthogonal with respect to the weight funtion%
\begin{equation*}
W_{{\beta },\mu }(x,t)=(t^{2}-\Vert x\Vert ^{2})^{\mu -\frac{1}{2}}t^{\beta
}e^{-t},\quad \mu >-\tfrac{1}{2},\quad {\beta }>-d,
\end{equation*}%
on the unbounded cone%
\begin{equation*}
{\mathbb{V}}^{d+1}=\left \{ (x,t):\Vert x\Vert \leq t,\, \,x\in {\mathbb{R}}%
^{d},\, \,0\leq t<\infty \right \} .
\end{equation*}%
For $\beta =0,$ the Laguerre polynomials on the cone are also eigenfunctions
of the following second order linear PDE with eigenvalues which depend only
on the degree of the polynomials%
\begin{equation}
\mathcal{D}_{\mu }u=-nu \label{equ-Laguerre}
\end{equation}%
where $D_{\mu }$\textbf{\ }denotes\textbf{\ }the second order linear
differential operator
\begin{equation*}
D_{\mu }=t(\Delta _{x}+\partial _{t}^{2})+2\left \langle x,\nabla _{x}\right
\rangle \partial _{t}-\left \langle x,\nabla _{x}\right \rangle +\left( 2\mu
+d-t\right) \partial _{t},~\mu >-1/2,
\end{equation*}%
where $\bigtriangledown _{x}$ and $\Delta _{x}$ denote the operators acting
on the $x$ variable.

\subsubsection{OPs on the surface of the cone}

Orthogonal polynomials on the surface of the cone $\mathbb{V}_{0}^{d+1}$,
denoted by%
\begin{equation*}
\mathbb{V}_{0}^{d+1}:=\left \{ \left( x,t\right) \in \mathbb{V}^{d+1}:\left
\Vert x\right \Vert =t,\text{ }0\leq t\leq b\right \}
\end{equation*}%
are defined with respect to a bilinear form given by%
\begin{equation*}
\left \langle f,g\right \rangle _{w}:=b_{w}\int_{\mathbb{V}%
_{0}^{d+1}}f\left( x,t\right) g\left( x,t\right) w\left( t\right) \mathrm{d}%
\sigma \left( x,t\right)
\end{equation*}%
where $w\left( t\right) $ is a nonnegative function on the interval $\left[
0,\infty \right) $ so that $\int \limits_{0}^{\infty }t^{d-1}w\left(
t\right) \mathrm{d}t<\infty $, $\mathrm{d}\sigma $ is the Lebesgue measure
on $\mathbb{V}_{0}^{d+1}$ and $b_{w}$ denotes a normalization constant
defined such that $\left \langle 1,1\right \rangle =1.$ The bilinear form ${%
\langle }.,.{\rangle }_{w}$ is an inner product on the space $\mathbb{R}%
\left[ x,t\right] /\left \langle \left \Vert x\right \Vert ^{2}-t^{2}\right
\rangle $ where we denote by $\mathbb{R}\left[ x,t\right] /\left \langle
q\right \rangle $ the space of polynomials in the variables $x$ and $t$
considered modulo the ideal $\left \langle q\right \rangle $ generated by $q.
$ Let ${\mathcal{V}}_{n}({\mathbb{V}}_{0}^{d+1},w)$ be the space of
orthogonal polynomials with respect to the inner product ${\langle }.,.{%
\rangle }_{w}.$ Then,
\begin{equation*}
\dim {\mathcal{V}}_{n}({\mathbb{V}}_{0}^{d+1},w)=\binom{n+d-1}{n}+\binom{%
n+d-2}{n-1}.
\end{equation*}%
As in the case of the solid cone, it is possible to define families of
orthogonal polynomials on the surface of the cone. One example of such
polynomials is Laguerre polynomials on the surface of the cone. For the
choose $b=\infty $ and $w_{\beta }\left( t\right) =t^{\beta }e^{-t},~\beta
>-d,$ consider the inner product on $\mathbb{R}\left[ x,t\right] /\left
\langle \left \Vert x\right \Vert ^{2}-t^{2}\right \rangle $%
\begin{equation*}
\left \langle f,g\right \rangle _{\beta }:=b_{\beta }\int_{\mathbb{V}%
_{0}^{d+1}}f\left( x,t\right) g\left( x,t\right) t^{\beta }e^{-t}\mathrm{d}%
\sigma \left( x,t\right) .
\end{equation*}%
Assume that the set $\left \{ Y_{l}^{m}:1\leq l\leq \mathcal{H}%
_{m}^{d}\right \} $ is an orthonormal basis of $\mathcal{H}_{m}^{d}.$ In
this case, an orthogonal basis for ${\mathcal{V}}_{n}({\mathbb{V}}%
_{0}^{d+1},w_{\beta })$ can be expressed in terms of Laguerre polynomials by%
\begin{equation*}
S_{n,m,l}^{\beta ,L}\left( x,t\right) =L_{n-m}^{2m+\beta +d-1}\left(
t\right) Y_{l}^{m}\left( x\right) ,~0\leq m\leq n,~1\leq l\leq \dim \mathcal{%
H}_{m}^{d}.
\end{equation*}%
For $\beta =-1,$ these polynomials are eigenfunctions of the second order
differential equation%
\begin{equation}
\left( t\partial _{t}^{2}+\left( d-1-t\right) \partial _{t}+t^{-1}\Delta
_{0}^{\left( x\right) }\right) u=-nu \label{sur-denk}
\end{equation}%
where $\Delta _{0}^{\left( x\right) }$ denotes the Laplace-Beltrami operator
in variable $x\in \mathbb{S}^{d-1},~d\geq 2$ (see \cite{X3}).

\section{Finite orthogonal polynomials in and on a cone}

The finite orthogonal structures in a solid cone will be considered in the
first subsection and the finite structures on the surface of the cone will be
presented in the second subsection.

\subsection{Finite orthogonal polynomials on a solid cone}

On the cone $\mathbb{V}^{d+1},$ orthogonal polynomials are orthogonal with
respect to the weight function%
\begin{equation}
W_{\mu }(x,t)=w(t)(t^{2}-\Vert x\Vert ^{2})^{\mu -\frac{1}{2}},\quad \mu >-%
\tfrac{1}{2}  \label{eq:Wmu-cone}
\end{equation}%
where $w$ is a weight function on ${\mathbb{R}}^{+}$ and these polynomials
satisfy the inner product
\begin{equation*}
{\langle }f,g{\rangle }_{\mu }:=b_{\mu }\int_{{\mathbb{V}}%
^{d+1}}f(x,t)g(x,t)W_{\mu }(x,t)\mathrm{d}x\mathrm{d}t,
\end{equation*}%
where $b_{\mu }$ is a normalization constant so that ${\langle }1,1{\rangle }%
_{\mu }=1.$ The normalization constant $b_{\mu }$ is given as in (\ref%
{nor-cone}).

On the cone $\mathbb{V}^{d+1}$, we investigate finite orthogonal polynomial
systems associated with the weight function $W_{\mu }(x,t)$, where two different choices of the function $w(t)$ are considered. In the
first setting, taking $w_{p,q}\left(t\right) =\frac{t^{q}}{\left( 1+t\right) ^{p+q}}$, we obtain the first finite family on the cone. In the second setting, choosing $w_{p}\left( t\right) =t^{-p}e^{-1/t}$, we derive the second finite family on the cone. While the first family consists of eigenfunctions of a second order differential operator whose corresponding eigenvalues depend only on the total degree of the polynomials, the second family satisfies a difference-differential equation.

\subsection{First finite class on the cone}

We consider the weight function
\begin{equation*}
W_{p,q,\mu }\left( x,t\right) =(t^{2}-\Vert x\Vert ^{2})^{\mu -\frac{1}{2}}%
\frac{t^{q}}{\left( 1+t\right) ^{p+q}}
\end{equation*}%
on the solid cone ${\mathbb{V}}^{d+1}$ and the corresponding inner product
defined by
\begin{equation*}
\left \langle f,g\right \rangle _{p,q,\mu }=b_{p,q,\mu }^{M}\int_{{\mathbb{V}%
}^{d+1}}f\left( x,t\right) g\left( x,t\right) W_{p,q,\mu }\left( x,t\right)
\mathrm{d}x\mathrm{d}t,
\end{equation*}%
where $b_{p,q,\mu }^{M}$ is the normalization constant and it is defined by $%
b_{p,q,\mu }^{M}=b_{\mu }^{\mathbb{B}}\times c_{p-2\mu -d+1,q+2\mu +d-1},$ where $%
c_{p,q}$ and $b_{\mu }^{\mathbb{B}}$ are defined by (\ref{Def:normcons M})
and (\ref{norm-ball}), respectively. Let ${\mathcal{V}}_{n}({\mathbb{V}}%
^{d+1},W_{p,q,\mu })$ be the space of orthogonal polynomials of degree $n$
with respect to the inner product $\left \langle .,.\right \rangle _{p,q,\mu
}. $ We can build the first class of finite polynomials on the cone in terms
of the first class of the univariate finite orthogonal polynomials $%
M_{n}^{(p,q)}(x)$ and orthogonal polynomials on the unit ball as follows:

\begin{proposition}
\label{Mpoly_def} Let ${\mathbb{P}}_{m}=\{P_{\mathbf{k}}:|{\mathbf{k}}|=m\}$
be an orthonormal basis of ${\mathcal{V}}_{m}^{d}({\mathbb{B}}^{d},{\mathsf{w%
}}_{\mu })$ for $m\leq n$ . Let $\alpha =\mu +\frac{d-1}{2}.$ Define%
\begin{equation}
M_{\mathbf{k,}n}^{\left( p,q,\mu \right) }\left( x,t\right) =M_{n-m}^{\left(
p-2\alpha -2m,q+2\alpha +2m\right) }\left( t\right) t^{m}P_{\mathbf{k}%
}^{m}\left( \frac{x}{t}\right) ,\text{\quad }\left \vert \mathbf{k}\right
\vert =m,~0\leq m\leq n.  \label{polM on the cone}
\end{equation}%
Then $\left \{ M_{\mathbf{k,}n}^{\left( p,q,\mu \right) }\left( x,t\right)
:\left \vert \mathbf{k}\right \vert =m,~\,0\leq m\leq n\right \} $ is an
orthogonal basis of ${\mathcal{V}}_{n}^{d}(\mathbb{V}^{d+1},W_{p,q,\mu })$
and it satisfies the orthogonality relation%
\begin{equation*}
\left \langle M_{\mathbf{k,}n}^{\left( p,q,\mu \right) },M_{\mathbf{k}%
^{\prime },n^{\prime }}^{\left( p,q,\mu \right) }\right \rangle _{p,q,\mu
}=H_{m,n}^{\left( p,q,\mu \right) }\delta _{n,n^{\prime }}\delta
_{m,m^{\prime }}\delta _{\mathbf{k},\mathbf{k}^{\prime }}
\end{equation*}%
under the restrictions $p>2N+2\mu +d$ and $q>-2\mu -d,~N=\max \left \{
n,n^{\prime }\right \} $ where $H_{m,n}^{\left( p,q,\mu \right) }$ denotes
the norm square of $M_{\mathbf{k,}n}^{\left( p,q,\mu \right) }$ given by
\begin{equation*}
H_{m,n}^{\left( p,q,\mu \right) }=\frac{c_{p-2\alpha ,q+2\alpha }}{%
c_{p-2\alpha -2m,q+2\alpha +2m}}h_{n-m}^{\left( p-2\alpha -2m,q+2\alpha
+2m\right) },
\end{equation*}%
where $c_{p,q}$ and $h_{n}^{(p,q)}$ are defined as in (\ref{Def:normcons M})
and (\ref{Def:norm M}), respectively.
\end{proposition}

\begin{proof}
It follows from the definition of the polynomials $M_{\mathbf{k,}n}^{\left(
p,q,\mu \right) }$%
\begin{eqnarray*}
&&\left \langle M_{\mathbf{k,}n}^{\left( p,q,\mu \right) },M_{\mathbf{k}%
^{\prime },n^{\prime }}^{\left( p,q,\mu \right) }\right \rangle _{p,q,\mu }
\\
&=&b_{p,q,\mu }^{M}\int_{\mathbb{V}^{d+1}}M_{n-m}^{\left( p-2\alpha
-2m,q+2m+2\alpha \right) }\left( t\right) M_{n^{\prime }-m^{\prime
}}^{\left( p-2\alpha -2m^{\prime },q+2\alpha +2m^{\prime }\right) }\left(
t\right) t^{m+m^{\prime }} \\
&&\times P_{\mathbf{k}}\left( \frac{x}{t}\right) P_{\mathbf{k}^{\prime
}}\left( \frac{x}{t}\right) (t^{2}-\Vert x\Vert ^{2})^{\mu -\frac{1}{2}}%
\frac{t^{q}}{\left( 1+t\right) ^{p+q}}\mathrm{d}x\mathrm{d}t \\
&=&c_{p-2\alpha ,q+2\alpha }\int \limits_{0}^{\infty }M_{n-m}^{\left(
p-2\alpha -2m,q+2\alpha +2m\right) }\left( t\right) M_{n^{\prime }-m^{\prime
}}^{\left( p-2\alpha -2m^{\prime },q+2\alpha +2m^{\prime }\right) }\left(
t\right) \frac{t^{q+2\alpha +m+m^{\prime }}}{\left( 1+t\right) ^{p+q}}%
\mathrm{d}t \\
&&\times b_{\mu }^{\mathbb{B}}\int_{{\mathbb{B}}^{d}}P_{\mathbf{k}}\left(
y\right) P_{\mathbf{k}^{\prime }}\left( y\right) \left( 1-\left \Vert
y\right \Vert ^{2}\right) ^{\mu -\frac{1}{2}}\mathrm{d}y.
\end{eqnarray*}%
Since $P_{\mathbf{k}}\left( y\right) $ is an orthonormal basis and $%
M_{n-m}^{\left( p-2\alpha -2m,q+2\alpha +2m\right) }\left( t\right) $ is
finitely orthogonal for $p>2n+2\mu +d$ and $q>-2\mu -d$, we obtain%
\begin{eqnarray*}
&&\left \langle M_{\mathbf{k,}n}^{\left( p,q,\mu \right) },M_{\mathbf{k}%
^{\prime },n^{\prime }}^{\left( p,q,\mu \right) }\right \rangle _{p,q,\mu }
\\
&=&c_{p-2\alpha ,q+2\alpha }\int \limits_{0}^{\infty }M_{n-m}^{\left(
p-2\alpha -2m,q+2\alpha +2m\right) }\left( t\right) M_{n^{\prime
}-m}^{\left( p-2\alpha -2m,q+2\alpha +2m\right) }\left( t\right) \\
&&\times \frac{t^{q+2\alpha +2m}}{\left( 1+t\right) ^{p+q}}\mathrm{d}%
t~\delta _{m,m^{\prime }}\delta _{\mathbf{k},\mathbf{k}^{\prime }} \\
&=&\frac{c_{p-2\alpha ,q+2\alpha }}{c_{p-2\alpha -2m,q+2\alpha +2m}}%
h_{n-m}^{\left( p-2\alpha -2m,q+2\alpha +2m\right) }~\delta _{n,n^{\prime
}}\delta _{m,m^{\prime }}\delta _{\mathbf{k},\mathbf{k}^{\prime }}.
\end{eqnarray*}
\end{proof}
For instance, let us consider $d=1,~p=202,~q=0$ and $\mu =0$. Two-variable
polynomials $\left \{ M_{\mathbf{k,}n}^{\left( 202,0,0\right) }\right \}
_{n=0}^{100}$ , $0\leq m\leq n,$\ are finitely orthogonal with respect to
the weight function $W_{202,0,0}\left( x,t\right) =(t^{2}-x^{2})^{-\frac{1}{2%
}}\frac{1}{\left( 1+t\right) ^{202}}$ on the cone. Indeed,%
\begin{eqnarray*}
&&\int \limits_{{\mathbb{V}}^{2}}~M_{\mathbf{k,}n}^{\left( 202,0,0\right)
}(x,t)M_{\mathbf{k}^{\prime },n^{\prime }}^{\left( 202,0,0\right)
}(x,t)(t^{2}-x^{2})^{-\frac{1}{2}}\frac{1}{\left( 1+t\right) ^{202}}\mathrm{d%
}x\mathrm{d}t \\
&=&\frac{c_{202,0}}{c_{202-2m,2m}}h_{n-m}^{\left( 202-2m,2m\right) }\delta
_{n,n^{\prime }}\delta _{m,m^{\prime }}\delta _{\mathbf{k},\mathbf{k}%
^{\prime }}.
\end{eqnarray*}%
For $d=1$, then these polynomials become $M_{m,n}^{\left( p,q,\mu \right)
}\left( x,t\right) =M_{n-m}^{\left( p-2m-2\mu ,q+2m+2\mu \right) }\left(
t\right) t^{m}C_{m}^{(\mu) }\left( \frac{x}{t}\right) $ where $\,0\leq m\leq
n. $ Here are some examples of these polynomials:%
\begin{eqnarray*}
M_{0,0}^{\left( p,q,\mu \right) }\left( x,t\right) &=&M_{0}^{\left( p-2\mu
,q+2\mu \right) }\left( t\right) C_{0}^{(\mu) }\left( \frac{x}{t}\right) =1, \\
M_{0,1}^{\left( p,q,\mu \right) }\left( x,t\right) &=&M_{1}^{\left( p-2\mu
,q+2\mu \right) }\left( t\right) C_{0}^{(\mu) }\left( \frac{x}{t}\right)
=\left( p-2\mu -2\right) t-\left( q+2\mu +1\right), \\
M_{1,1}^{\left( p,q,\mu \right) }\left( x,t\right) &=&M_{0}^{\left( p-2-2\mu
,q+2+2\mu \right) }\left( t\right) tC_{1}^{(\mu) }\left( \frac{x}{t}\right)
=2\mu x ,\\
M_{0,2}^{\left( p,q,\mu \right) }\left( x,t\right) &=&M_{2}^{\left( p-2\mu
,q+2\mu \right) }\left( t\right) C_{0}^{(\mu) }\left( \frac{x}{t}\right) \\
&=&\left( p-2\mu -4\right) \left( p-2\mu -3\right) t^{2}-2\left( p-2\mu
-3\right) \left( q+2\mu +2\right) t \\
&&+\left( q+2\mu +2\right) \left( q+2\mu +1\right), \\
M_{1,2}^{\left( p,q,\mu \right) }\left( x,t\right) &=&M_{1}^{\left( p-2-2\mu
,q+2+2\mu \right) }\left( t\right) tC_{1}^{(\mu) }\left( \frac{x}{t}\right) =%
\left[ \left( p-2\mu -2\right) t-\left( q+2\mu +1\right) \right] 2\mu x, \\
M_{2,2}^{\left( p,q,\mu \right) }\left( x,t\right) &=&M_{0}^{\left( p-4-2\mu
,q+4+2\mu \right) }\left( t\right) t^{2}C_{2}^{(\mu) }\left( \frac{x}{t}%
\right) =2\mu \left( \mu +1\right) x^{2}-\mu t^{2} \\
&&\vdots
\end{eqnarray*}

Our next theorem shows that for $q=0$ the polynomials in ${\mathcal{V}}_{n}({%
\mathbb{V}}^{d+1},W_{p,0,\mu })$ are eigenfunctions of a second order
differential operator.

\begin{theorem}
\label{equ-cone} For $n=0,1,2,...,$ every $v\in {\mathcal{V}}_{n}({\mathbb{V}%
}^{d+1},W_{p,0,\mu })$ satisfies the differential equation
\begin{eqnarray}
&&\left[ t(1+t)\partial t^{2}+2\left( 1+t\right) \left \langle
x,\bigtriangledown _{x}\right \rangle \partial t+\left( d+2\mu +\left(
-p+2\mu +d+1\right) t\right) \partial t\right.  \label{difeqM} \\
&&+\left. t\left( 1+t\right) \bigtriangleup _{x}-\left( t^{2}\bigtriangleup
_{x}-\left \langle x,\bigtriangledown _{x}\right \rangle ^{2}-\left( 2\mu
+d-p\right) \left \langle x,\bigtriangledown _{x}\right \rangle \right)
\right] v  \notag \\
&=&n\left( n-p+2\mu +d\right) v,  \notag
\end{eqnarray}%
where $\bigtriangleup _{x}$ and $\bigtriangledown _{x}$ indicate that
operators are acting on $x$ variable.

\begin{proof}
From the definition of $M_{\mathbf{k,}n}^{\left( p,q,\mu \right) }\left(
x,t\right) $, it is sufficient to set the result for $v=M_{\mathbf{k,}%
n}^{\left( p,0,\mu \right) }\left( x,t\right) $ in (\ref{polM on the cone}).
For the sake of brevity, let's consider%
\begin{equation*}
v\left( x,t\right) =f\left( t\right) H\left( x,t\right) ,
\end{equation*}%
where $f\left( t\right) =M_{n-m}^{\left( p-2m-2\mu -d+1,2m+2\mu +d-1\right)
}\left( t\right) $ and $H\left( x,t\right) =t^{m}P_{\mathbf{k}}^{m}\left(
\frac{x}{t}\right) .$ The differential equation for the finite polynomial $%
M_{n}^{\left( p,q\right) }$ shows that $f$ satisfies the following equation
\begin{eqnarray}
&&t\left( 1+t\right) f^{\prime \prime }+\left( \left( -p+2m+2\mu +d+1\right)
t+\left( d+2m+2\mu \right) \right) f^{\prime }  \label{difeq_g} \\
&=&\left( n-m\right) \left( n+m-p+2\mu +d\right) f.  \notag
\end{eqnarray}%
Since the polynomials $P_{\mathbf{k}}\left( x\right) $ satisfy the
differential equation (\ref{diffop}), the polynomials $H\left( x,t\right) $
satisfy the following differential equation%
\begin{equation}
\left( t^{2}\Delta _{x}-\left \langle x,\nabla _{x}\right \rangle
^{2}-\left( 2\mu +d-1\right) \left \langle x,\nabla _{x}\right \rangle
\right) H\left( x,t\right) =-m\left( m+2\mu +d-1\right) H\left( x,t\right) .
\label{fsatisfy}
\end{equation}%
Also, since $H\left( x,t\right) $ is a homogeneous function of degree $m$ in
its variables, this identity holds%
\begin{equation}
\left( t\frac{\partial }{\partial t}+\left \langle x,\nabla _{x}\right
\rangle \right) H\left( x,t\right) =mH\left( x,t\right) ,
\label{euleridentity}
\end{equation}%
from which, by taking the derivative with respect to $t,$ it follows that%
\begin{equation}
t\frac{\partial ^{2}}{\partial t^{2}}H\left( x,t\right) +\left \langle
x,\nabla _{x}\right \rangle \frac{\partial }{\partial t}H\left( x,t\right)
=\left( m-1\right) \frac{\partial }{\partial t}H\left( x,t\right) .
\label{dereuleridentity}
\end{equation}%
A direct computation of the derivatives of $v$ leads to%
\begin{eqnarray*}
&&t\left( 1+t\right) v_{tt}+2\left( 1+t\right) \left \langle x,\nabla
_{x}\right \rangle v_{t} \\
&=&t\left( 1+t\right) f^{\prime \prime }\left( t\right) H\left( x,t\right)
+2\left( 1+t\right) f^{\prime }\left( t\right) \left( t\frac{\partial }{%
\partial t}+\left \langle x,\nabla _{x}\right \rangle \right) H\left(
x,t\right) \\
&&+\left( 1+t\right) f\left( t\right) \left( 2\left \langle x,\nabla
_{x}\right \rangle \frac{\partial }{\partial t}H\left( x,t\right) +t\frac{%
\partial ^{2}}{\partial t^{2}}H\left( x,t\right) \right) .
\end{eqnarray*}%
Applying (\ref{euleridentity}) to the second term on the right hand side, it
follows from (\ref{difeq_g})%
\begin{eqnarray*}
&&t\left( 1+t\right) v_{tt}+2\left( 1+t\right) \left \langle x,\nabla
_{x}\right \rangle v_{t} \\
&=&-\left( d+2\mu +\left( -p+2\mu +d+1\right) t\right) f^{\prime }\left(
t\right) H\left( x,t\right) \\
&&+\left( n-m\right) \left( n+m-p+2\mu +d\right) v \\
&&+\left( 1+t\right) f\left( t\right) \left( 2\left \langle x,\nabla
_{x}\right \rangle \frac{\partial }{\partial t}H\left( x,t\right) +t\frac{%
\partial ^{2}}{\partial t^{2}}H\left( x,t\right) \right) .
\end{eqnarray*}%
By adding and subtracting the term $\left( d+2\mu +\left( -p+2\mu
+d+1\right) t\right) f\left( t\right) \frac{\partial }{\partial t}H\left(
x,t\right) $ in the last equation, we get
\begin{eqnarray}
&&t\left( 1+t\right) v_{tt}+2\left( 1+t\right) \left \langle x,\nabla
_{x}\right \rangle v_{t}+\left( d+2\mu +\left( -p+2\mu +d+1\right) t\right)
v_{t}  \label{eq} \\
&=&\left( n-m\right) \left( n+m-p+2\mu +d\right) v+\left( 1-p\right)
tf\left( t\right) \frac{\partial }{\partial t}H\left( x,t\right)  \notag \\
&&+\left( 1+t\right) f\left( t\right) \left[ t\frac{\partial ^{2}}{\partial
t^{2}}H\left( x,t\right) +\left( 2\left \langle x,\nabla _{x}\right \rangle
+2\mu +d\right) \frac{\partial }{\partial t}H\left( x,t\right) \right] .
\notag
\end{eqnarray}%
Using (\ref{fsatisfy}), (\ref{euleridentity}), and (\ref{dereuleridentity}),
the term inside the square bracket on the right hand side can be written as
follows
\begin{eqnarray}
&&t\frac{\partial ^{2}}{\partial t^{2}}H\left( x,t\right) +\left( 2\left
\langle x,\nabla _{x}\right \rangle +2\mu +d\right) \frac{\partial }{%
\partial t}H\left( x,t\right)  \notag \\
&=&\left[ \left( 2\mu +d+m-1+\left \langle x,\nabla _{x}\right \rangle
\right) \frac{\partial }{\partial t}H\left( x,t\right) \right]  \notag \\
&=&t^{-1}\left( -\left \langle x,\nabla _{x}\right \rangle ^{2}-\left( 2\mu
+d-1\right) \left \langle x,\nabla _{x}\right \rangle +m\left( 2\mu
+m+d-1\right) \right) H\left( x,t\right)  \notag \\
&=&-t\Delta _{x}H\left( x,t\right) .  \label{A}
\end{eqnarray}%
If we use this equality and also apply the identity (\ref{euleridentity}) in
the second term on the right hand side of (\ref{eq}), we conclude that%
\begin{eqnarray*}
&&t\left( 1+t\right) v_{tt}+2\left( 1+t\right) \left \langle x,\nabla
_{x}\right \rangle v_{t}+\left( d+2\mu +\left( -p+2\mu +d+1\right) t\right)
v_{t} \\
&&+t\left( 1+t\right) \Delta _{x}v+\left( 1-p\right) \left \langle x,\nabla
_{x}\right \rangle v \\
&=&\left( \left( n-m\right) \left( n+m-p+2\mu +d\right) +\left( 1-p\right)
m\right) v.
\end{eqnarray*}%
By means of the following identity
\begin{equation}
\left( n-m\right) \left( n+m-p+2\mu +d\right) v+m\left( 1-p\right) v
=n\left(
n+2\mu +d-p\right) v-m\left( m+2\mu +d-1\right) v  \label{iden}
\end{equation}
and taking into account (\ref{fsatisfy}), we complete the proof.
\end{proof}
\end{theorem}

\begin{remark}
When $q\neq 0,$ the finite polynomials $M_{\mathbf{k,}n}^{\left( p,q,\mu
\right) }\left( x,t\right) $ on the cone satisfy a differential equation,
but the eigenvalues are determined by both $m$ and $n$.
\end{remark}

\begin{theorem}
\label{rec-M} Let $M_{\mathbf{k,}n}^{\left( p,q,\mu \right) }\in {\mathcal{V}%
}_{n}({\mathbb{V}}^{d+1},W_{p,q,\mu })$. The polynomials $M_{\mathbf{k,}%
n}^{\left( p,q,\mu \right) }\left( x,t\right) $ satisfy the following
three-term recurrence relation
\begin{eqnarray}
M_{\mathbf{k,}n+1}^{\left( p,q,\mu \right) }\left( x,t\right) &=&\left \{
\frac{\left( p-2n-2\mu -d\right) \left( p-2n-2\mu -d-1\right) }{\left(
p-m-n-2\mu -d\right) }t\right.  \notag \\
&&+\frac{2\left( n-m\right) \left( n-m+1\right) }{\left( p-2\mu
-d-2n+1\right) }  \notag \\
&&-\left. \frac{\left( p-2m-2\mu -d+1\right) \left( d+q+2\mu +2n\right) }{%
\left( p-2\mu -d-2n+1\right) }\right \} M_{\mathbf{k,}n}^{\left( p,q,\mu
\right) }\left( x,t\right)  \label{three-termrecur} \\
&&-\left \{ \frac{\left( n-m\right) \left( p-2\mu -d-2n-1\right) }{\left(
p-m-2\mu -d-n\right) }\right.  \notag \\
&&\times \left. \frac{\left( p+q-n+m\right) \left( d+q+2\mu +m+n-1\right) }{%
\left( p-2\mu -d-2n+1\right) }\right \} M_{\mathbf{k,}n-1}^{\left( p,q,\mu
\right) }\left( x,t\right) .  \notag
\end{eqnarray}
\end{theorem}

\begin{proof}
By substituting $n\rightarrow n-m$, $p\rightarrow p-2\alpha -2m$, $%
q\rightarrow q+2\alpha +2m$, and $x\rightarrow t$ in (\ref{recurrenceM}), we get%
\begin{eqnarray*}
&&M_{n+1-m}^{\left( p-2\alpha -2m,q+2\alpha +2m\right) }\left( t\right) \\
&=&\left \{ \frac{\left( p-2n-2\mu -d\right) \left( p-2n-2\mu -d-1\right) }{%
\left( p-m-n-2\mu -d\right) }t\right. \\
&&+\frac{2\left( n-m\right) \left( n-m+1\right) }{\left( p-2n-2\mu
-d+1\right) } \\
&&-\left. \frac{\left( p-2m-2\mu -d+1\right) \left( d+q+2n+2\mu \right) }{%
\left( p-2n-2\mu -d+1\right) }\right \} M_{n-m}^{\left( p-2\alpha
-2m,q+2\alpha +2m\right) }\left( t\right) \\
&&-\left \{ \frac{\left( n-m\right) \left( p-2n-2\mu -d-1\right) }{\left(
p-m-n-2\mu -d\right) }\right. \\
&&\times \left. \frac{\left( p+q+m-n\right) \left( q+m+n+2\mu +d-1\right) }{%
\left( p-2n-2\mu -d+1\right) }\right \} M_{n-1-m}^{\left( p-2\alpha
-2m,q+2\alpha +2m\right) }\left( t\right) ,
\end{eqnarray*}%
where $\alpha =\mu +\frac{d-1}{2}.$ Multiplying both sides by $t^{m}P_{%
\mathbf{k}}^{m}\left( \frac{x}{t}\right) $, we obtain (\ref{three-termrecur}%
).
\end{proof}

We now give a limit relation between the finite polynomials $M_{\mathbf{k,}%
n}^{\left( p,q,\mu \right) }\left( x,t\right) $ on the cone and Laguerre
polynomials ${\mathsf{L}}_{{\mathbf{k}},n}^{\beta ,\mu }(x,t)$ on the cone
as follows:

\begin{theorem}
\label{limit-cone} The limit relation is satisfied between $M_{\mathbf{k,}%
n}^{\left( p,q,\mu \right) }\left( x,t\right) $ and ${\mathsf{L}}_{{\mathbf{k%
}},n}^{q,\mu }(x,t)$%
\begin{equation*}
\underset{p\rightarrow \infty }{\lim }p^{m}M_{\mathbf{k,}n}^{\left( p,q,\mu
\right) }\left( \frac{x}{p},\frac{t}{p}\right) =\left( -1\right)
^{n-m}\left( n-m\right) !{\mathsf{L}}_{{\mathbf{k}},n}^{q,\mu }(x,t).
\end{equation*}
\end{theorem}

\begin{proof}
From the Rodrigues formula (\ref{M-Rod}) and the limit relation (\ref{limit}%
) we can write%
\begin{eqnarray*}
&&\underset{p\rightarrow \infty }{\lim }M_{n-m}^{\left( p-2\alpha
-2m,q+2\alpha +2m\right) }\left( \frac{t}{p}\right) \\
&=&\left( -1\right) ^{n-m}\underset{p\rightarrow \infty }{\lim }\frac{\left(
1+t/p\right) ^{p+q}}{t^{q+2m+2\alpha }}\frac{d^{n-m}\left( t^{n+m+q+2\alpha
}\left( 1+t/p\right) ^{n-m-p-q}\right) }{dt^{n-m}} \\
&=&\left( -1\right) ^{n-m}t^{-\left( q+2m+2\alpha \right) }e^{t}\frac{%
d^{n-m}\left( t^{n+m+q+2\alpha }e^{-t}\right) }{dt^{n-m}} \\
&=&\left( -1\right) ^{n-m}\left( n-m\right) !L_{n-m}^{q+2m+2\alpha }\left(
t\right) ,
\end{eqnarray*}%
from which we conclude that%
\begin{eqnarray*}
\underset{p\rightarrow \infty }{\lim }p^{m}M_{\mathbf{k,}n}^{\left( p,q,\mu
\right) }\left( \frac{x}{p},\frac{t}{p}\right) &=&\left( -1\right)
^{n-m}\left( n-m\right) !L_{n-m}^{q+2m+2\alpha }\left( t\right) t^{m}P_{%
\mathbf{k}}^{m}\left( \frac{x}{t}\right) \\
&=&\left( -1\right) ^{n-m}\left( n-m\right) !{\mathsf{L}}_{{\mathbf{k}}%
,n}^{q,\mu }(x,t).
\end{eqnarray*}
\end{proof}

\begin{remark}
Using this limit relation, we can give the differential equation for
Laguerre polynomials on the cone. Substituting $x\rightarrow \frac{x}{p}$ and%
$~t\rightarrow \frac{t}{p}$ in Theorem \ref{equ-cone}, and after multiplying
$p^{m-1}$ and letting $p\rightarrow \infty $ $,$ we obtain the differential
equation given by (\ref{equ-Laguerre}) for Laguerre polynomials $u={\mathsf{L}}_{{\mathbf{k}},n}^{0,\mu
}(x,t)$ on the cone
\begin{equation*}
\left( t(\Delta _{x}+\partial _{t}^{2})+2\left \langle x,\nabla _{x}\right
\rangle \partial _{t}-\left \langle x,\nabla _{x}\right \rangle +\left( 2\mu
+d-t\right) \partial _{t}\right) u=-nu.
\end{equation*}
\end{remark}

\begin{remark}
Replacing $x$ by $\frac{x}{p}$ and $~t$ by $\frac{t}{p}$ in Theorem \ref%
{rec-M}, and then multiplying $p^{m}$ and taking the limit as $p\rightarrow
\infty ,$ we arrive at following recurrence relation for Laguerre
polynomials ${\mathsf{L}}_{{\mathbf{k}},n}^{q,\mu }(x,t)$ on the cone%
\begin{equation*}
\left( n-m+1\right) {\mathsf{L}}_{{\mathbf{k}},n+1}^{q,\mu }(x,t)=\left(
d+q+2\mu +2n-t\right) {\mathsf{L}}_{{\mathbf{k}},n}^{q,\mu }(x,t)-\left(
d+q+2\mu +m+n-1\right) {\mathsf{L}}_{{\mathbf{k}},n-1}^{q,\mu }(x,t).
\end{equation*}
\end{remark}

\subsection{Second finite class on the cone}

Let consider the weight function
\begin{equation*}
W_{\mu ,p}\left( x,t\right) =(t^{2}-\Vert x\Vert ^{2})^{\mu -\frac{1}{2}%
}t^{-p}e^{-1/t}
\end{equation*}%
on the solid cone
\begin{equation*}
{\mathbb{V}}^{d+1}=\left \{ (x,t):\Vert x\Vert \leq t,\, \,x\in {\mathbb{B}}%
^{d},\, \,0\leq t<\infty \right \}
\end{equation*}%
and define the inner product%
\begin{equation*}
\left \langle f,g\right \rangle _{\mu ,p}:=b_{\mu ,p}^{N}\int_{{\mathbb{V}}%
^{d+1}}f\left( x,t\right) g\left( x,t\right) (t^{2}-\Vert x\Vert ^{2})^{\mu -%
\frac{1}{2}}t^{-p}e^{-1/t}\mathrm{d}x\mathrm{d}t
\end{equation*}%
where
\begin{equation*}
b_{\mu ,p}^{N}=c_{p-2\mu -d+1}\times b_{\mu }^{\mathbb{B}}\text{ with }c_{p}=%
\frac{1}{\Gamma \left( p-1\right) }
\end{equation*}%
is the normalization constant such that $\left \langle 1,1\right \rangle
_{\mu ,p}=1.$

Let ${\mathcal{V}}_{n}({\mathbb{V}}^{d+1},W_{\mu ,p})$ be the space of
orthogonal polynomials of degree $n$ in $d+1$ variables with respect to the
inner product ${\langle }.,.{\rangle }_{\mu ,p}$. Similar to the first
class, the second class of finite polynomials on the cone is defined in
terms of the second class of finite orthogonal polynomials $N_{n}^{(p)}(x)$
in one variable and orthogonal polynomials on the unit ball as follows:

\begin{proposition}
\label{prop:OPconefinite} Let ${\mathbb{P}}_{m}=\{P_{\mathbf{k}}:|{\mathbf{k}%
}|=m\}$ be an orthonormal basis of ${\mathcal{V}}_{m}^{d}({\mathbb{B}}^{d},{%
\mathsf{w}}_{\mu })$ for $m\leq n$ . Define
\begin{equation*}
N_{\mathbf{k,}n}^{\left( \mu ,p\right) }\left( x,t\right) =N_{n-m}^{\left(
p-2\alpha -2m\right) }\left( t\right) t^{m}P_{\mathbf{k}}^{m}\left( \frac{x}{%
t}\right) ,\text{\quad }\left \vert \mathbf{k}\right \vert =m,~0\leq m\leq n
\end{equation*}%
where $\alpha =\mu +\frac{d-1}{2}{.}$ Then $\mathbb{N}_{n,m}=\left \{ N_{%
\mathbf{k,}n}^{\left( \mu ,p\right) }:\left \vert \mathbf{k}\right \vert
=m,~\,0\leq m\leq n\right \} $ is an orthogonal basis of ${\mathcal{V}}%
_{n}^{d}(\mathbb{V}^{d+1},W_{\mu ,p})$ and it satisfies the orthogonality
relation
\begin{equation*}
\left \langle N_{\mathbf{k,}n}^{\left( \mu ,p\right) },N_{\mathbf{k}^{\prime
},n^{\prime }}^{\left( \mu ,p\right) }\right \rangle _{\mu
,p}=H_{m,n}^{\left( \mu ,p\right) }\delta _{n,n^{\prime }}\delta
_{m,m^{\prime }}\delta _{\mathbf{k},\mathbf{k}^{\prime }}
\end{equation*}%
under the restriction $p>2N+2\mu +d,$\ $N=\max \left \{ n,n^{\prime }\right
\} $ where $H_{m,n}^{\left( \mu ,p\right) }$ denotes the norm square of $N_{%
\mathbf{k,}n}^{\left( \mu ,p\right) }$ given by
\begin{equation*}
H_{m,n}^{\left( \mu ,p\right) }=\frac{c_{p-2\alpha }}{c_{p-2\alpha -2m}}%
h_{n-m}^{\left( p-2\alpha -2m\right) },
\end{equation*}%
where $h_{n}^{\left( p\right) }$ is given by (\ref{normN}).
\end{proposition}

For example, let us consider $d=1,$ $\mu =0,~p=300.$ $\ $Two-variable
polynomials $\left \{ N_{\mathbf{k,}n}^{\left( 0,300\right) }\right \}
_{n=0}^{149}$ , $0\leq m\leq n,$\ are finitely orthogonal with respect to
the weight function $W_{0,300}\left( x,t\right) =\left( t^{2}-x^{2}\right)
^{-1/2}t^{-300}e^{-1/t}$ on the cone. Indeed, for $n,n^{\prime }\leq 149$ it
follows%
\begin{eqnarray*}
&&\int_{\mathbb{V}^{2}}~N_{\mathbf{k,}n}^{\left( 0,300\right) }\left(
x,t\right) ~N_{\mathbf{k}^{\prime },n^{\prime }}^{\left( 0,300\right)
}\left( x,t\right) \left( t^{2}-x^{2}\right) ^{-1/2}t^{-300}e^{-1/t}\mathrm{d%
}x\mathrm{d}t \\
&=&\frac{c_{300}}{c_{300-2m}}h_{n-m}^{\left( 300-2m\right) }\delta
_{n,n^{\prime }}\delta _{m,m^{\prime }}\delta _{\mathbf{k},\mathbf{k}%
^{\prime }}.
\end{eqnarray*}%
For $d=1$, then these polynomials become $N_{m,n}^{\left( \mu ,p\right)
}\left( x,t\right) =N_{n-m}^{\left( p-2\mu -2m\right) }\left( t\right)
t^{m}C_{m}^{(\mu) }\left( \dfrac{x}{t}\right) $ where $\,0\leq m\leq n.$ Here
are some examples of these polynomials:
\begin{eqnarray*}
N_{0,0}^{\left( \mu ,p\right) }\left( x,t\right) &=&N_{0}^{\left( p-2\mu
\right) }\left( t\right) C_{0}^{(\mu) }\left( \frac{x}{t}\right) =1, \\
N_{0,1}^{\left( \mu ,p\right) }\left( x,t\right) &=&N_{1}^{\left( p-2\mu
\right) }\left( t\right) C_{0}^{(\mu) }\left( \frac{x}{t}\right) =\left(
p-2\mu -2\right) t-1, \\
N_{1,1}^{\left( \mu ,p\right) }\left( x,t\right) &=&N_{0}^{\left( p-2\mu
-2\right) }\left( t\right) tC_{1}^{(\mu) }\left( \frac{x}{t}\right) =2\mu x, \\
N_{0,2}^{\left( \mu ,p\right) }\left( x,t\right) &=&N_{2}^{\left( p-2\mu
\right) }\left( t\right) C_{0}^{(\mu) }\left( \frac{x}{t}\right) =\left(
p-2\mu -4\right) \left( p-2\mu -3\right) t^{2} \\
&&\text{ \  \  \  \  \  \  \  \  \  \  \  \  \  \  \  \  \  \  \  \  \  \  \  \  \  \  \  \ }-2\left(
p-2\mu -3\right) t+1, \\
N_{1,2}^{\left( \mu ,p\right) }\left( x,t\right) &=&N_{1}^{\left( p-2\mu
-2\right) }\left( t\right) tC_{1}^{(\mu) }\left( \frac{x}{t}\right) =\left(
\left( p-2\mu -4\right) t-1\right) 2\mu x, \\
N_{2,2}^{\left( \mu ,p\right) }\left( x,t\right) &=&N_{0}^{\left( p-2\mu
-4\right) }\left( t\right) t^{2}C_{2}^{(\mu) }\left( \frac{x}{t}\right) =2\mu
\left( 1+\mu \right) x^{2}-\mu t^{2} \\
&&\vdots
\end{eqnarray*}

Unlike other orthogonal polynomials families on the cone, the finite
polynomials $N_{\mathbf{k,}n}^{\left( \mu ,p\right) }$ do not satisfy a
second order partial differential equation. Instead, they satisfy a
differential-difference equation which we will give in the next theorem.

\begin{theorem}
For $m=0,1,\ldots ,$ let $\mu >-\frac{1}{2},$ $p>2n+2\mu +d$ and $\mathbb{P}%
_{m}=\left \{ P_{\mathbf{k}}:\left \vert \mathbf{k}\right \vert =m\right \} $
be an orthonormal basis of $\mathcal{V}_{m}\left( \mathbb{B}^{d},{\mathsf{w}}%
_{\mu }\right) .$ Then the polynomials $N_{\mathbf{k,}n}^{\left( \mu
,p\right) }$ satisfy the following differential-difference equation%
\begin{equation*}
L_{\mu ,p}\left[ N_{\mathbf{k,}n}^{\left( \mu ,p\right) }\right] =n\left(
n+2\mu +d-p\right) ~N_{\mathbf{k,}n}^{\left( \mu ,p\right) }-\left(
n-m\right) \left( p-2\mu -m-n-d\right) ~N_{\mathbf{k,}n-1}^{\left( \mu
,p-2\right) }
\end{equation*}

where $L_{\mu ,p}$ is given by
\begin{equation*}
L_{\mu ,p}=t^{2}\frac{\partial ^{2}}{\partial t^{2}}+\left( 2t\left \langle
x,\nabla _{x}\right \rangle +\left( 1+2\mu +d-p\right) t\right) \frac{%
\partial }{\partial t}+\left( 2\mu +d-p\right) \left \langle x,\nabla
_{x}\right \rangle +\left \langle x,\nabla _{x}\right \rangle ^{2},
\end{equation*}%
where $\bigtriangledown _{x}$ denotes the gradient operator acting on $x$
variable.
\end{theorem}

\begin{proof}
According to the\ definition of $N_{\mathbf{k,}n}^{\left( \mu ,p\right) }$,
for the simplicity, let $v\left( x,t\right) =f\left( t\right) H\left(
x,t\right) $ for%
\begin{equation}
f\left( t\right) =N_{n-m}^{\left( p-2\mu -2m-d+1\right) }\left( t\right) ,%
\text{\qquad }H\left( x,t\right) =t^{m}P_{\mathbf{k}}\left( \frac{x}{t}%
\right) .  \label{changevariable}
\end{equation}%
From the equation (\ref{Dif-N}), it is seen that the function $f$ satisfies
\begin{eqnarray}
&&t^{2}f^{\prime \prime }\left( t\right) +\left( \left( 1+2\mu
+2m+d-p\right) t+1\right) f^{\prime }\left( t\right)  \label{gsatisfy} \\
&&-\left. \left( n-m\right) \left( n+m+2\mu +d-p\right) f\left( t\right)
=0.\right.  \notag
\end{eqnarray}%
Since $H\left( x,t\right) $ is a homogeneous polynomial of degree $m$, the
identity (\ref{euleridentity}) holds and also $H\left( x,t\right) $
satisfies the differential equation (\ref{fsatisfy}). By taking the
derivatives of $v,$ we directly obtain
\begin{eqnarray*}
t^{2}v_{tt}+\left( 2t\left \langle x,\nabla _{x}\right \rangle -1\right)
v_{t} &=&t^{2}f^{\prime \prime }\left( t\right) H\left( x,t\right) \\
&&+2tf^{\prime }\left( t\right) \left( t\frac{\partial }{\partial t}H\left(
x,t\right) +\left \langle x,\nabla _{x}\right \rangle H\left( x,t\right)
\right) \\
&&+tf\left( t\right) \left( t\frac{\partial ^{2}}{\partial t^{2}}H\left(
x,t\right) +2\left \langle x,\nabla _{x}\right \rangle \frac{\partial }{%
\partial t}H\left( x,t\right) \right) \\
&&-f\left( t\right) \frac{\partial }{\partial t}H\left( x,t\right)
-f^{\prime }\left( t\right) H\left( x,t\right) .
\end{eqnarray*}%
Substituting (\ref{euleridentity}) into the second term on the right hand
side and making use of (\ref{gsatisfy}), we arrive at
\begin{eqnarray*}
&&t^{2}v_{tt}+\left( 2t\left \langle x,\nabla _{x}\right \rangle -1\right)
v_{t}+\left( \left( 1+2\mu +d-p\right) t+1\right) v_{t} \\
&=&\left( n-m\right) \left( n+m+2\mu +d-p\right) v+\left( 1-p\right)
tf\left( t\right) \frac{\partial }{\partial t}H\left( x,t\right) -f^{\prime
}\left( t\right) H\left( x,t\right) \\
&&+tf\left( t\right) \left[ t\frac{\partial ^{2}}{\partial t^{2}}H\left(
x,t\right) +\left( 2\left \langle x,\nabla _{x}\right \rangle +2\mu
+d\right) \frac{\partial }{\partial t}H\left( x,t\right) \right] .
\end{eqnarray*}%
From the equality (\ref{A}), it follows%
\begin{eqnarray*}
t^{2}v_{tt}+2t\left \langle x,\nabla _{x}\right \rangle v_{t}+\left( 1+2\mu
+d-p\right) tv_{t} &=&\left( n-m\right) \left( n+m+2\mu +d-p\right) v \\
&&+\left( 1-p\right) tf\left( t\right) \frac{\partial }{\partial t}H\left(
x,t\right) \\
&&-f^{\prime }\left( t\right) H\left( x,t\right) -t^{2}\Delta _{x}v,
\end{eqnarray*}%
from which, using (\ref{euleridentity}) again in the second term on the
right hand side yields
\begin{eqnarray*}
&&t^{2}v_{tt}+2t\left \langle x,\nabla _{x}\right \rangle v_{t}+\left(
1+2\mu +d-p\right) tv_{t}+t^{2}\Delta _{x}v+\left( 1-p\right) \left \langle
x,\nabla _{x}\right \rangle v \\
&=&\left( n-m\right) \left( n+m+2\mu +d-p\right) v+m\left( 1-p\right)
v-f^{\prime }\left( t\right) H\left( x,t\right) .
\end{eqnarray*}%
Considering the identity (\ref{iden}) and using (\ref{fsatisfy}) again, we
deduce that
\begin{eqnarray*}
&&t^{2}v_{tt}+2t\left \langle x,\nabla _{x}\right \rangle v_{t}+\left(
1+2\mu +d-p\right) tv_{t}+\left( 2\mu +d-p\right) \left \langle x,\nabla
_{x}\right \rangle v+\left \langle x,\nabla _{x}\right \rangle ^{2}v \\
&=&n\left( n+2\mu +d-p\right) v-f^{\prime }\left( t\right) H\left(
x,t\right) .
\end{eqnarray*}%
Finally, using the recurrence relation (\ref{Nderrel}) for the function $f$
in the second term of the right hand side, we conclude that
\begin{eqnarray*}
&&t^{2}\frac{\partial ^{2}}{\partial t^{2}}\left[ N_{\mathbf{k,}n}^{\left(
\mu ,p\right) }\right] +2t\left \langle x,\nabla _{x}\right \rangle \frac{%
\partial }{\partial t}\left[ N_{\mathbf{k,}n}^{\left( \mu ,p\right) }\right]
+\left( 1+2\mu +d-p\right) t\frac{\partial }{\partial t}\left[ N_{\mathbf{k,}%
n}^{\left( \mu ,p\right) }\right] \\
&&+\left( 2\mu +d-p\right) \left \langle x,\nabla _{x}\right \rangle \left[
N_{\mathbf{k,}n}^{\left( \mu ,p\right) }\right] +\left \langle x,\nabla
_{x}\right \rangle ^{2}\left[ N_{\mathbf{k,}n}^{\left( \mu ,p\right) }\right]
\\
&=&n\left( n+2\mu +d-p\right) \left[ N_{\mathbf{k,}n}^{\left( \mu ,p\right) }%
\right] -\left( n-m\right) \left( p-2\mu -m-n-d\right) ~N_{\mathbf{k,}%
n-1}^{\left( \mu ,p-2\right) },
\end{eqnarray*}%
which completes the proof.
\end{proof}

Using the recurrence relation for the univariate finite polynomials $%
N_{n}^{\left( p\right) }$, we can give a recurrence relation for $N_{\mathbf{%
k,}n}^{\left( \mu ,p\right) }\in {\mathcal{V}}_{n}({\mathbb{V}}^{d+1},W_{\mu
,p})$ in the next theorem.

\begin{theorem}
The polynomials $N_{\mathbf{k,}n}^{\left( \mu ,p\right) }$satisfy the
following three-term recurrence relation%
\begin{eqnarray}
N_{\mathbf{k,}n+1}^{\left( \mu ,p\right) }\left( x,t\right) &=&\left \{
\frac{\left( p-2\mu -2n-d-1\right) \left( p-2\mu -2n-d\right) }{\left(
p-2\mu -m-n-d\right) }t\right.  \notag \\
&&-\left. \frac{\left( p-2\mu -2m-d+1\right) \left( p-2\mu -2n-d\right) }{%
\left( p-2\mu -m-n-d\right) \left( p-2\mu -2n-d+1\right) }\right \} ~N_{%
\mathbf{k,}n}^{\left( \mu ,p\right) }\left( x,t\right)  \label{Nrecreloncone}
\\
&&-\frac{\left( n-m\right) \left( p-2\mu -2n-d-1\right) }{\left( p-2\mu
-m-n-d\right) \left( p-2\mu -2n-d+1\right) }~N_{\mathbf{k,}n-1}^{\left( \mu
,p\right) }\left( x,t\right) .  \notag
\end{eqnarray}
\end{theorem}

\begin{proof}
If we take $n\rightarrow n-m,$ $x\rightarrow t,$ $p\rightarrow
p-2\alpha -2m$ in (\ref{Nrecrel}), we deduce that
\begin{eqnarray*}
N_{n+1-m}^{\left( p-2\alpha -2m\right) }\left( t\right) &=&\left \{ \frac{%
\left( p-2\mu -2n-d-1\right) \left( p-2\mu -2n-d\right) }{\left( p-2\mu
-m-n-d\right) }t\right. \\
&&-\left. \frac{\left( p-2\mu -2m-d+1\right) \left( p-2\mu -2n-d\right) }{%
\left( p-2\mu -m-n-d\right) \left( p-2\mu -2n-d+1\right) }\right \}
N_{n-m}^{\left( p-2\alpha -2m\right) }\left( t\right) \\
&&-\frac{\left( n-m\right) \left( p-2\mu -2n-d-1\right) }{\left( p-2\mu
-m-n-d\right) \left( p-2\mu -2n-d+1\right) }N_{n-1-m}^{\left( p-2\alpha
-2m\right) }\left( t\right)
\end{eqnarray*}
where $\alpha =\mu +\frac{d-1}{2}.$ If we multiple both sides by $t^{m}P_{%
\mathbf{k}}\left( \frac{x}{t}\right) ,$ we arrive at the given
result.\bigskip
\end{proof}

\subsection{Finite orthogonal polynomials on the surface of a cone}

We will consider finite orthogonal polynomials on the surface of the cone
which is denoted by%
\begin{equation*}
\mathbb{V}_{0}^{d+1}:=\left \{ \left( x,t\right) \in \mathbb{V}^{d+1}:\left
\Vert x\right \Vert =t,\text{ }0\leq t<\infty \right \} .
\end{equation*}%
Let $w\left( t\right) $ be a nonnegative function on the interval $\left[
0,\infty \right) $ so that $\int \limits_{0}^{\infty }t^{d-1}w\left(
t\right) \mathrm{d}t<\infty $. We define the inner product on the surface of
the $\mathbb{V}_{0}^{d+1}$ as follows%
\begin{equation*}
\left \langle f,g\right \rangle _{w}:=b_{w}\int_{\mathbb{V}%
_{0}^{d+1}}f\left( x,t\right) g\left( x,t\right) w\left( t\right) \mathrm{d}%
\sigma \left( x,t\right)
\end{equation*}%
where $\mathrm{d}\sigma $ is the Lebesgue measure on the surface of the cone
and $b_{w}$ denotes a normalization constant defined such that $\left
\langle 1,1\right \rangle =1.$ Let ${\mathcal{V}}_{n}({\mathbb{V}}%
_{0}^{d+1},w)$ be the space of orthogonal polynomials with respect to the
inner product ${\langle }.,.{\rangle }_{w},$ which is well defined on the
space of polynomials in $\left( x,t\right) $ variables modulo the polynomial
idea generated by $\left \Vert x\right \Vert ^{2}-t^{2}.$

In this section, as in the previous part of the paper, we will examine two
classes of finite orthogonal polynomials on the surface of the cone. Unlike
case of the on the cone, we will use spherical harmonics on the surface of
the cone. We will present the orthogonality relations and obtain the
differential equations satisfied by them. Let's start with first family.

\subsection{First finite class on the surface of the cone}

We consider the weight function $w_{p,q}=\frac{t^{q}}{\left( 1+t\right)
^{p+q}}$ on the surface of the cone $\mathbb{V}_{0}^{d+1}$
\begin{equation*}
\mathbb{V}_{0}^{d+1}:=\left \{ \left( x,t\right) \in \mathbb{V}^{d+1}:\left
\Vert x\right \Vert =t,\text{ }0\leq t<\infty \right \}
\end{equation*}
and define the inner product
\begin{equation*}
\left \langle f,g\right \rangle _{p,q}=b_{p,q}\int \limits_{\mathbb{V}%
_{0}^{d+1}}f\left( x,t\right) g\left( x,t\right) t^{q}\left( 1+t\right)
^{-\left( p+q\right) }\mathrm{d}\sigma \left( x,t\right)
\end{equation*}
on the surface of the cone $\mathbb{V}_{0}^{d+1}$.
By changing variables with $x=\xi t,$ the integral on the surface of the
cone can be written as follows
\begin{eqnarray*}
\int_{\mathbb{V}_{0}^{d+1}}f\left( x,t\right) \mathrm{d}\sigma \left(
x,t\right) &=&\int \limits_{0}^{\infty }\int_{\left \Vert x\right \Vert
=t}f\left( x,t\right) \mathrm{d}\sigma \left( x,t\right) \\
&=&\int \limits_{0}^{\infty }t^{d-1}\int_{\mathbb{S}^{d-1}}f\left( t\xi ,t\right) \mathrm{d}\sigma \left( \xi \right) \mathrm{d}t,
\end{eqnarray*}%
from which, the normalization constant is given by $b_{p,q}=\frac{1}{\omega _{d}}c_{p-d+1,q+d-1}$ such that $\left \langle 1,1\right \rangle _{p,q}=1$ where $\omega _{d}$ denotes the surface area of $\mathbb{S}^{d-1}$ and $c_{p,q}$ is defined as in (\ref{Def:normcons M}).

As stated at the beginning of this section, a basis of $\mathcal{V}%
_{n}^{d}\left( \mathbb{V}_{0}^{d+1},w_{p,q}\right) $ can be expressed in
terms of the finite orthogonal polynomials $M_{n}^{\left( p,q\right) }$ and
spherical harmonics.

\begin{proposition}
For $m=0,1,\ldots $ let $\left \{ Y_{l}^{m}:1\leq l\leq \dim \mathcal{H}%
_{m}^{d}\right \} $ denote an orthonormal basis of $\mathcal{H}_{m}^{d}$.
For $0\leq m\leq n,$ $1\leq l\leq \dim \mathcal{H}_{m}^{d},$ define
\begin{equation}
S_{n,m,l}^{p,q,M}\left( x,t\right) =M_{n-m}^{\left( p-2m-d+1,q+2m+d-1\right)
}\left( t\right) Y_{l}^{m}\left( x\right) ,\text{\quad }p>2N+d,\text{ }%
q>-d,~N=\max \left \{ n,n^{\prime }\right \} .  \label{Mpoly onsurface}
\end{equation}%
Then the family $\left \{ S_{n,m,l}^{p,q,M}:0\leq m\leq n,\text{ }1\leq
l\leq \dim \mathcal{H}_{m}^{d}\right \} $ forms an orthogonal basis of $%
\mathcal{V}_{n}^{d}\left( \mathbb{V}_{0}^{d+1},w_{p,q}\right) $ for $p>2N+d,$
$q>-d,~N=\max \left \{ n,n^{\prime }\right \} $. More precisely,%
\begin{equation}
\left \langle S_{n,m,l}^{p,q,M},S_{n^{\prime },m^{\prime },l^{\prime
}}^{p,q,M}\right \rangle _{p,q}=H_{m,n}^{p,q,M}\delta _{n,n^{\prime }}\delta
_{l,l^{\prime }}\delta _{m,m^{\prime }},  \label{Mpoly_orthogonal}
\end{equation}%
where
\begin{equation}
H_{m,n}^{p,q,M}=\left \langle S_{n,m,l}^{p,q,M},S_{n,m,l}^{p,q,M}\right
\rangle _{p,q}=\frac{c_{p-d+1,q+d-1}}{c_{p-2m-d+1,q+2m+d-1}}%
h_{n-m}^{(p-2m-d+1,q+2m+d-1)}  \label{Mpoly_orthonormal}
\end{equation}%
where $h_{n}^{\left( p,q\right) }$ denotes the norm square of the
polynomials $M_{n}^{\left( p,q\right) }$given by (\ref{Def:norm M}).
\end{proposition}

\begin{proof}
From the definition of $S_{n,m,l}^{p,q,M},$ we can write
\begin{eqnarray*}
\left \langle S_{n,m,l}^{p,q,M},S_{n^{\prime },m^{\prime },l^{\prime }}^{%
\text{ }p,q,M}\right \rangle &=&b_{p,q}\int_{\mathbb{V}%
_{0}^{d+1}}S_{n,m,l}^{p,q,M}\left( x,t\right) S_{n^{\prime },m^{\prime
},l^{\prime }}^{\text{ }p,q,M}\left( x,t\right) t^{q}\left( 1+t\right)
^{-\left( p+q\right) }\mathrm{d}\sigma \left( x,t\right) \\
&=&\frac{1}{\varpi _{d}}c_{p-d+1,q+d-1}\int \limits_{0}^{\infty }\int_{\left
\Vert x\right \Vert =t}M_{n-m}^{\left( p-2m-d+1,q+2m+d-1\right) }\left(
t\right) M_{n^{\prime }-m^{\prime }}^{\left( p-2m^{\prime }-d+1,q+2m^{\prime
}+d-1\right) }\left( t\right) \\
&&\times Y_{l}^{m}\left( x\right) Y_{l^{\prime }}^{m^{\prime }}\left(
x\right) t^{q}\left( 1+t\right) ^{-\left( p+q\right) }\mathrm{d}\sigma
\left( x,t\right) .
\end{eqnarray*}

Then using change of variable $x=t\xi ,$ we derive
\begin{eqnarray*}
&=&c_{p-d+1,q+d-1}\int \limits_{0}^{\infty }M_{n-m}^{\left(
p-2m-d+1,q+2m+d-1\right) }\left( t\right) M_{n^{\prime }-m^{\prime
}}^{\left( p-2m^{\prime }-d+1,q+2m^{\prime }+d-1\right) }\left( t\right)
t^{m+m^{\prime }+d-1+q}\left( 1+t\right) ^{-\left( p+q\right) }dt \\
&&\times \frac{1}{\varpi _{d}}\int_{\mathbb{S}^{d-1}}Y_{l}^{m}\left( \xi
\right) Y_{l^{\prime }}^{m^{\prime }}\left( \xi \right) \mathrm{d}\sigma
\left( \xi \right) .
\end{eqnarray*}

Since $Y_{l}^{m}$ is orthonormal, it follows%
\begin{eqnarray*}
\left \langle S_{n,m,l}^{p,q,M},S_{n^{\prime },m^{\prime },l^{\prime }}^{%
\text{ }p,q,M}\right \rangle &=&\delta _{m,m^{\prime }}\delta _{l,l^{\prime
}}~c_{p-d+1,q+d-1}\int \limits_{0}^{\infty }M_{n-m}^{\left(
p-2m-d+1,q+2m+d-1\right) }\left( t\right) M_{n^{\prime }-m}^{\left(
p-2m-d+1,q+2m+d-1\right) }\left( t\right) \\
&&\times \frac{t^{q+2m+d-1}}{\left( 1+t\right) ^{p+q}}dt \\
&=&\frac{c_{p-d+1,q+d-1}}{c_{p-2m-d+1,q+2m+d-1}}%
h_{n-m}^{(p-2m-d+1,q+2m+d-1)}\delta _{n,n^{\prime }}\delta _{m,m^{\prime
}}\delta _{l,l^{\prime }}.
\end{eqnarray*}
\end{proof}

As in the case of the solid cone, the first family $S_{n,m,l}^{p,q,M}\left(
x,t\right) $ of finite orthogonal polynomials on the surface of the cone are
eigenfunctions of a second order differential operator only when $q=-1.$

\begin{theorem}
\label{equ-first}Let $d\geq 2$. Every $v\in \mathcal{V}_{n}^{d}\left(
\mathbb{V}_{0}^{d+1},w_{p,-1}\right) $ satisfies the differential equation
\begin{equation}
\left[ t\left( 1+t\right) \partial t^{2}+\left( d-1+\left( -p+d+1\right)
t\right) \partial t+t^{-1}\Delta _{0}^{\left( x\right) }\right] v=n\left(
n-p+d\right) v,  \label{difeqMpoly_onsurface}
\end{equation}%
where $\Delta _{0}^{\left( x\right) }$ denotes the Laplace-Beltrami operator
in $x\in \mathbb{S}^{d-1}.$

\begin{proof}
We set up the result for $v\left( x,t\right) =S_{n,m,l}^{p,q,M}\left(
x,t\right) $ in (\ref{Mpoly onsurface}). For $x=t\xi ,$ let's write it as
follows%
\begin{equation*}
S_{n,m,l}^{p,q,M}\left( x,t\right) =f\left( t\right) t^{m}Y_{l}^{m}\left(
\xi \right)
\end{equation*}%
where $f\left( t\right) =M_{n-m}^{\left( p-2m-d+1,q+2m+d-1\right) }\left(
t\right) .$ Using the differential equation (\ref{equ-M}) for the finite
orthogonal polynomials $M_{n}^{\left( p,q\right) }\left( x\right) $, it is
seen that the polynomial $g\left( t\right) =f\left( t\right) t^{m}$
satisfies the equation%
\begin{eqnarray*}
&&t\left( 1+t\right) g^{\prime \prime }\left( t\right) +\left( q+d+\left(
-p+d+1\right) t\right) g^{\prime }\left( t\right)  \\
&=&n\left( n-p+d\right) g\left( t\right) +m\left( m+q+d-1\right)
t^{-1}g\left( t\right) .
\end{eqnarray*}%
This equation is also satisfied for $v\left( x,t\right) =g\left( t\right)
Y_{l}^{m}\left( \xi \right) $
\begin{equation*}
\left[ t\left( 1+t\right) \partial t^{2}+\left( q+d+\left( -p+d+1\right)
t\right) \partial t\right] v=n\left( n-p+d\right) v+m\left( m+q+d-1\right)
t^{-1}g\left( t\right) Y_{l}^{m}\left( \xi \right) .
\end{equation*}%
For $q=-1$, the term $m\left( m+q+d-1\right) $ in this equation gives the
eigenvalue of $Y_{l}^{m}$ for the operator $-\Delta _{0}$ and the last term reduces to
\begin{equation*}
m\left( m+d-2\right) t^{-1}g\left( t\right) Y_{l}^{m}=-t^{-1}g\left(
t\right) \Delta _{0}^{\left( x\right) }Y_{l}^{m}=-t^{-1}\Delta _{0}^{\left(
x\right) }S_{n,m,l}^{p,q,M}.
\end{equation*}%
Thus, the proof is completed.
\end{proof}
\end{theorem}

\begin{remark}
We notice that the polynomial $S_{n,m,l}^{p,q,M}\left( x,t\right) $
satisfies a differential equation when $q\neq -1,$ but the equation depends
on both $m$ and $n.$
\end{remark}

Similar to Theorem \ref{limit-cone}, we present a limit relation between the
finite polynomials $S_{n,m,l}^{p,q,M}\left( x,t\right) $ on the surface of
the cone and Laguerre polynomials $S_{n,m,l}^{q,L}\left( x,t\right) $ on the
surface of the cone using the relation (\ref{limit}) as follows:

\begin{theorem}
The limit relation between $S_{n,m,l}^{p,q,M}\left( x,t\right) $ and $%
S_{n,m,l}^{q,L}\left( x,t\right) $ is given by%
\begin{equation*}
\underset{p\rightarrow \infty }{\lim }S_{n,m,l}^{p,q,M}\left( x,\frac{t}{p}%
\right) =\left( -1\right) ^{n-m}\left( n-m\right) !L_{n-m}^{q+2m+d-1}\left(
t\right) Y_{l}^{m}\left( x\right) =\left( -1\right) ^{n-m}\left( n-m\right)
!S_{n,m,l}^{q,L}\left( x,t\right) .
\end{equation*}
\end{theorem}

\begin{remark}
Substituting $t\rightarrow \frac{t}{p}$ in Theorem \ref{equ-first} and then
taking the limit as $p\rightarrow \infty ,$ we obtain that%
\begin{equation*}
\left( t\partial _{t}^{2}+\left( d-1-t\right) \partial _{t}+t^{-1}\Delta
_{0}^{\left( x\right) }\right) S_{n,m,l}^{-1,L}\left( x,t\right)
=-nS_{n,m,l}^{-1,L}\left( x,t\right) ,
\end{equation*}%
which is the differential equation given by (\ref{sur-denk}) for the Laguerre polynomials $%
S_{n,m,l}^{-1,L}\left( x,t\right) $ on the surface of the cone when $q=-1.$
\end{remark}

\subsection{Second finite class on the surface of the cone}

If we choose the weight function $w_{p}\left( t\right) =t^{-p}e^{-\frac{1}{t}%
}$, $p>d,$ then the inner product becomes
\begin{equation*}
\left \langle f,g\right \rangle _{p}=b_{p}\int_{\mathbb{V}_{0}^{d+1}}f\left(
x,t\right) g\left( x,t\right) t^{-p}e^{-\frac{1}{t}}\mathrm{d}\sigma \left(
x,t\right)
\end{equation*}%
on the surface of the cone $\mathbb{V}_{0}^{d+1}$ where its normalization constant is given by
\begin{equation*}
b_{p}=\frac{1}{\varpi _{d}}\frac{1}{\int \limits_{0}^{\infty }t^{d-p-1}e^{-%
\frac{1}{t}}\mathrm{d}t}=\frac{1}{\varpi _{d}}\frac{1}{\Gamma \left(
p-d\right) }
\end{equation*}%
such that $\left \langle 1,1\right \rangle _{p}=1$ where $\varpi _{d}$ is
the surface area of $\mathbb{S}^{d-1}.$

Similar to the first class, we can express a basis of the $\mathcal{V}%
_{n}^{d}\left( \mathbb{V}_{0}^{d+1},w_{p}\right) $ in terms of the finite
orthogonal polynomials $N_{n}^{\left( p\right) }$ and spherical harmonics in
the next proposition.

\begin{proposition}
For $m=0,1,\ldots $ let $\left \{ Y_{l}^{m}:1\leq l\leq \dim \mathcal{H}%
_{m}^{d}\right \} $ denote an orthonormal basis of $\mathcal{H}_{m}^{d}$.
For $0\leq m\leq n,$ $1\leq l\leq \dim \mathcal{H}_{m}^{d},$ define
\begin{equation}
S_{n,m,l}^{p,N}\left( x,t\right) =N_{n-m}^{\left( p-2m-d+1\right) }\left(
t\right) Y_{l}^{m}\left( x\right) ,\text{\quad }p>2N+d,\text{ }N=\max \left
\{ n,n^{\prime }\right \} .  \label{Defsur1}
\end{equation}%
Then the family $\left \{ S_{n,m,l}^{p,N}:0\leq m\leq n,\text{ }1\leq l\leq
\dim \mathcal{H}_{m}^{d}\right \} $ forms an orthogonal basis of $\mathcal{V}%
_{n}^{d}\left( \mathbb{V}_{0}^{d+1},w_{p}\right) $ for $p>2N+d,$ $N=\max
\left \{ n,n^{\prime }\right \} .$ Indeed, we have
\begin{equation*}
\left \langle S_{n,m,l}^{p,N},S_{n^{\prime },m^{\prime },l^{\prime }}^{\text{
}p,N}\right \rangle _{p}=\frac{\left( n-m\right) !\left( p-n-m-d\right) !}{%
\left( p-2n-d\right) \Gamma \left( p-d\right) }\delta _{n,n^{\prime }}\delta
_{m,m^{\prime }}\delta _{l,l^{\prime }}.
\end{equation*}
\end{proposition}

As in the case of the solid cone, this family of finite orthogonal
polynomials on the surface of the cone satisfies the following theorem.

\begin{theorem}
Let $d\geq 2.$ $S_{n,m,l}^{p,N}\in \mathcal{V}_{n}^{d}\left( \mathbb{V}%
_{0}^{d+1},w_{p}\right) $ satisfies the difference-differential equation%
\begin{eqnarray*}
&&t^{2}\frac{\partial ^{2}}{\partial t^{2}}\left[ S_{n,m,l}^{p,N}\right]
+\left( 1-p+d\right) t\frac{\partial }{\partial t}\left[ S_{n,m,l}^{p,N}%
\right] \\
&=&n\left( n-p+d\right) S_{n,m,l}^{p,N}+\left( m-n\right) \left(
p-n-m-d\right) S_{n-1,m,l}^{p-2,N},
\end{eqnarray*}
\end{theorem}

\begin{proof}
We derive the result for the function $S_{n,m,l}^{p,N}(x,t)$ defined in (\ref%
{Defsur1}). By applying the substitution $x=\xi t$ the polynomial $%
S_{n,m,l}^{p,N}(x,t)$ can be decomposed into the product $f\left( t\right)
t^{m}Y_{l}^{m}\left( \xi \right) $ where $f\left( t\right) =N_{n-m}^{\left(
p-2m-d+1\right) }\left( t\right) .$ From the differential equation (\ref%
{Dif-N}), the polynomial $f\left( t\right) $ satisfies%
\begin{equation*}
t^{2}f^{\prime \prime }\left( t\right) +\left( \left( 1-p+2m+d\right)
t+1\right) f^{\prime }\left( t\right) -\left( n-m\right) \left(
n+m-p+d\right) f\left( t\right) =0.
\end{equation*}
Using the recurrence relation (\ref{Nderrel}) of univariate finite
polynomial $N_{n}^{\left( p\right) }\left( t\right) ,$ we can write this
equation as follows
\begin{eqnarray*}
&&t^{2}f^{\prime \prime }\left( t\right) +\left( 1-p+2m+d\right) tf^{\prime
}\left( t\right) +\left( n-m\right) \left( p-n-m-d\right) N_{n-m-1}^{\left(
p-2m-d-1\right) }\left( t\right) \\
&=&\left( n-m\right) \left( n+m-p+d\right) f\left( t\right) ,
\end{eqnarray*}%
from which the polynomial $g\left( t\right) =f\left( t\right) t^{m}$
satisfies
\begin{equation*}
t^{2}g^{\prime \prime }\left( t\right) +\left( 1-p+d\right) tg^{\prime
}\left( t\right) =n\left( n-p+d\right) g\left( t\right) -\left( n-m\right)
\left( p-n-m-d\right) N_{n-m-1}^{\left( p-2m-d-1\right) }\left( t\right)
t^{m}.
\end{equation*}
The obtained equation is also valid for $S_{n,m,l}^{p,N}\left( x,t\right)
=g\left( t\right) Y_{l}^{m}\left( \xi \right) $
\begin{eqnarray*}
t^{2}\frac{\partial ^{2}}{\partial t^{2}}\left[ S_{n,m,l}^{p,N}\right]
+\left( 1-p+d\right) t\frac{\partial }{\partial t}\left[ S_{n,m,l}^{p,N}%
\right] &=&n\left( n-p+d\right) S_{n,m,l}^{p,N} \\
&&-\left( n-m\right) \left( p-n-m-d\right) \left[ N_{n-m-1}^{\left(
p-2m-d-1\right) }\left( t\right) t^{m}Y_{l}^{m}\left( \xi \right) \right] .
\end{eqnarray*}%
From the definition of $S_{n,m,l}^{p,N}\left( x,t\right) ,$ we obtain that%
\begin{eqnarray*}
&&t^{2}\frac{\partial ^{2}}{\partial t^{2}}\left[ S_{n,m,l}^{p,N}\right]
+\left( 1-p+d\right) t\frac{\partial }{\partial t}\left[ S_{n,m,l}^{p,N}%
\right] \\
&=&n\left( n-p+d\right) S_{n,m,l}^{p,N}-\left( n-m\right) \left(
p-n-m-d\right) S_{n-1,m,l}^{p-2,N},
\end{eqnarray*}
which completes the proof.
\end{proof}

\textbf{Acknowledgements}
Not applicable.\\

\textbf{Authors' contributions}
All authors contributed equally to this work. All authors have read and approved the final manuscript.\\

\textbf{Funding} No funding.\\

\textbf{Data availability}
Data sharing is not applicable to this article as no data sets were generated or analyzed during the current study.\\

\section*{Declarations}
\textbf{Conflict of interest} The authors declare no competing interests.\\

\textbf{Ethical Approval} Not applicable.


\begin{thebibliography}{99}
\bibitem{Bochner} S. Bochner, Uber Sturm-Liouvillesche Polynomsysteme, Math.
Z., 29 (1929), 730--736.

\bibitem{DaiX} F. Dai and Y. Xu, Approximation Theory and Harmonic Analysis
on Spheres and Balls\textit{. }Spinger Monographs in Mathematics. Springer,
New York, 2013.

\bibitem{DX} C. Dunkl and Y. Xu, Orthogonal polynomials of several
variables. Encyclopedia of Mathematics and its Applications, 155. Cambridge
Univ. Press, Cambridge, 2014.

\bibitem{E1} E. G\"{u}ldo\u{g}an Lekesiz, Finite biorthogonal M matrix polynomials,
Comput. Math. Math. Phys., 4 (66) (2026).

\bibitem{GA} E. G\"{u}ldo\u{g}an and R. Akta\c{s}, Some limit relationships
between some two-variable finite and infinite sequences of orthogonal
polynomials, J. Difference Equ. Appl., 27 (12) (2021), 1692-1722.

\bibitem{EA} E. G\"{u}ldo\u{g}an Lekesiz and I. Area, Some new families of
finite orthogonal polynomials in two variables, Axioms 12 (10) (2023), 932.

\bibitem{GAM1} E. G\"{u}ldo\u{g}an, R. Akta\c{s} and M. Masjed-Jamei, On
finite classes of two-variable orthogonal polynomials, Bull.Iran. Math.
Soc., 46 (2020), 1163--1194.

\bibitem{GAM2} E. G\"{u}ldo\u{g}an, R. Akta\c{s} and M. Masjed-Jamei,
Fourier transforms of some finite bivariate orthogonal polynomials,
Symmetry, 13 (2021), 452.

\bibitem{GCO} E. G\"{u}ldo\u{g}an Lekesiz, B. \c{C}ekim and M.A. \"{O}%
zarslan, Finite bivariate biorthogonal N-Konhauser polynomials, J.
Difference Equ. Appl., 31 (9) (2025), 1237-1262.

\bibitem{M1} M. Masjed Jamei, Three finite classes of hypergeometric
orthogonal polynomials and their application in functions approximation,
Integr. Trans. Spec. Funct., 13 (2) (2002), 169-190.

\bibitem{M2} M. Masjed-Jamei, Classical orthogonal polynomials with weight
function $\left( \left( ax+b\right) ^{2}+\left( cx+d\right) ^{2}\right)
^{-p}\exp \left( q\arctan \left( \frac{ax+b}{cx+d}\right) \right) ;~-\infty
<x<\infty $ and a generalization of T and F distributions. Integr. Trans.
Spec. Funct. 15(2) (2004), 137--153.

\bibitem{Tez} M. Masjed-Jamei, Some New Classes of Orthogonal Polynomials and Special Functions: A Symmetric Generalization of Sturm-Liouville Problems and its Consequences, Ph.D thesis, University of Kassel, Germany, 2006.

\bibitem{MSA} M. Masjed-Jamei, F. Soleyman, I. Area and J.J. Nieto, Two
finite q-Sturm liouville problems and their orthogonal polynomial solutions,
Filomat, 32(1) (2018), 231--244.

\bibitem{PX} P. Petrushev and Y. Xu, Localized polynomial frames on the
ball, Constr. Approx., 27 (2008), Paper No. 121, 148 pp.

\bibitem{PX2} M. Pinar and Y. Xu, Best polynomial approximation on the unit
ball, IMA J. Numer. Anal. 38 (2018), 1209-1228.

\bibitem {25}E. D. Rainville, Special functions, 1st ed. Chelsea Publishing Co., Bronx, N.Y., 1971.

\bibitem{SMA} F. Soleyman, M. Masjed-Jamei and I. Area, A finite class of
q-orthogonal polynomials corresponding to inverse gamma distribution, Anal.
Math. Phys., 7 (2017), 479--492.

\bibitem{Sz} G. Szeg\H{o}, Orthogonal polynomials. American Mathematical
Society Colloquium Publication 23, American Mathematical Society,
Providence, RI, 1975.

\bibitem{X1} Y. Xu, Orthogonal polynomials for a family of product weight
functions on the spheres, Canadian J. Math., 49 (1997), 175-192.

\bibitem{X2} Y. Xu, Orthogonal polynomials and cubature formulae on the
balls, simplices, and spheres, J. Comp. Appl. Math., 127 (2001), 349-368.

\bibitem{X3} Y. Xu, Orthogonal polynomials and Fourier orthogonal series on
a cone, J. Fourier Anal. Appl., 26 (2020), Paper No. 36, 42 pp.

\bibitem{X5} Y. Xu, Laguerre expansions on conic domains, J. Fourier Anal.
Appl. 27 (4) (2021), Paper No. 64, 36 pp.

\bibitem{X6} Y. Xu, Orthogonal structure and orthogonal series in and on a
double cone or a hyperboloid, Trans. Amer. Math. Soc., 374 (2021), 3603-3657.

\bibitem{X7} Y. Xu, Fourier orthogonal series on a paraboloid, J. d'Analyse
Math. 149 (2023), 251-279.

\bibitem{X8} Y. Xu, Orthogonal polynomials on domains of revolution, Studies
in Applied Math. 153 (2) (2024), Paper No. e12703, 38 pp.

\bibitem{X9} Y. Xu, Approximation and orthogonality on fully symmetric
domains, Studies in Applied Math. no. 1, Paper No. e70171, 28 pp.
\end{thebibliography}
\end{document}